\documentclass[reqno]{amsart}

\usepackage{geometry}
\usepackage{hyperref}
\usepackage{color}
\usepackage{esint}
\usepackage{cancel}
\usepackage[color=red!65!yellow]{todonotes}

\title[Two-phase flow with singular energy]{Construction of a two-phase flow with singular energy by gradient flow methods}

\author{Cl\'ement Canc\`es}
\address{Inria, Univ. Lille, CNRS, UMR 8524 - Laboratoire Paul Painlev\'e, F-59000 Lille (\href{mailto:clement.cances@inria.fr}{\tt clement.cances@inria.fr})}
\author{Daniel Matthes}
\address{Zentrum f\"ur Mathematik, Technische Universit\"at M\"unchen,
  85747 Garching, Germany (\href{mailto:matthes@ma.tum.de}{\tt
    matthes@ma.tum.de})}
\thanks{This research was supported by the DFG Collaborative Research Center TRR 109, ‘Discretization in Geometry and Dynamics.’}

\newtheorem{thm}{Theorem}
\newtheorem{prp}{Proposition}
\newtheorem{lem}{Lemma}
\newtheorem{cor}{Corollary}
\newtheorem{asm}{Assumption}
\newtheorem{rmk}{Remark}

\newcommand{\wass}{\mathbf W}
\newcommand{\altnrg}{\mathcal E}
\newcommand{\nrg}{\mathbf E}
\newcommand{\nrgone}{\mathbf E_1}
\newcommand{\ent}{\mathbf H}
\newcommand{\spc}{{{\mathbf X}_\text{mass}}}
\newcommand{\dst}{\mathbf d}

\newcommand{\setR}{{\mathbb R}}
\newcommand{\setN}{{\mathbb N}}
\newcommand{\dv}{\operatorname{div}}
\newcommand{\id}{\mathrm{id}}
\newcommand{\dd}{{\,\mathrm d}}
\newcommand{\dn}{{\mathrm d}}
\newcommand{\nml}{\mathbf{n}}
\newcommand{\eins}{\mathbf{1}}
\newcommand{\devil}{\mathbb{I}}
\newcommand{\eps}{\varepsilon}
\newcommand{\tr}{\mathrm{tr}}

\newcommand{\loc}{\text{loc}}
\newcommand{\bc}{{\boldsymbol{c}}}
\newcommand{\bq}{{\boldsymbol{q}}}
\newcommand{\bmu}{{\boldsymbol{\mu}}}

\newcommand{\err}{\boldsymbol{\epsilon}}
\newcommand{\brho}{\boldsymbol{\rho}}
\newcommand{\smallhalf}{\textstyle{\frac12}}

\newcommand{\Topt}{T_\text{opt}}
\newcommand{\opt}{\text{opt}}

\newcommand{\lbg}{{\mathcal L}}

\newcommand{\nonl}{{\mathfrak F}}

\begin{document}

\begin{abstract}
  We prove the existence of weak solutions to a system of two diffusion equations
  that are coupled by a pointwise volume constraint.
  The time evolution is given by gradient dynamics for a free energy functional.
  Our primary example is a model for the demixing of polymers,
  the corresponding energy is the one of Flory, Huggins and deGennes.
  Due to the non-locality in the equations, the dynamics considered here is qualitatively different
  from the one found in the formally related Cahn-Hilliard equations.
  
  Our angle of attack is from the theory of optimal mass transport,
  that is, we consider the evolution equations for the two components
  as two gradient flows in the Wasserstein distance with one joint energy functional
  that has the volume constraint built in.
  The main difference to our previous work \cite{CMN18} is the nonlinearity
  of the energy density in the gradient part,
  which becomes singular at the interface between pure and mixed phases.
\end{abstract}

\maketitle

\section{Introduction}
We show existence of non-negative solutions to the following coupled system of diffusion equations:
\begin{subequations}
  \label{eq:eqs}
  \begin{align}
    \label{eq:eqcont1}
    \partial_tc_1 &= \dv(m_1 c_1\nabla\mu_1), \\
    \label{eq:eqcont2}
    \partial_tc_2 &= \dv(m_2 c_2\nabla\mu_2), \\
    \label{eq:eqcons}
    c_1+c_2 &= 1,\\
    \label{eq:eqlagr}
    \mu_1-\mu_2 &= {-}f'(c_1)\Delta f(c_1) +  \chi {\textstyle{\left(\frac12 -  c_1\right)}},
  \end{align}
\end{subequations}
on a bounded and convex domain $\Omega\subset\setR^d$
in the plane ($d=2$) or physical space ($d=3$)
with smooth boundary $\partial\Omega$.
Solutions are subject to no-flux and homogeneous Neumann boundary conditions
\begin{subequations}
  \label{eq:bc}
  \begin{align}
    \label{eq:noflux}
    \nml\cdot(c_1\nabla\mu_1)=\nml\cdot(c_2\nabla\mu_2)=0, \\
    \label{eq:neumann}
    \nml\cdot\nabla c_1=\nml\cdot\nabla c_2=0
  \end{align}
\end{subequations}
on $\partial\Omega$ and to the initial conditions
\begin{align}
  \label{eq:ic}
  c_1(0)=c_1^0,\quad c_2(0)=c_2^0,
\end{align}
with initial data $c_1^0,c_2^0:\Omega\to[0,1]$ satisfying the constraint \eqref{eq:eqcons}.
The mobility coefficients $m_1,m_2>0$ and the parameter $\chi>0$ are given constants,
and the function $f:[0,1]\to\setR$ in \eqref{eq:eqlagr} is assumed to satisfy:
\begin{asm}
  \label{asm}
  $f$ is continuous on $[0,1]$,
  it is smooth on $(0,1)$ with $f'(r)>0$ there,
  it satisfies $f'(r)\to+\infty$ for $r\downarrow0$ and for $r\uparrow1$,
  and the function $1/(f')^2$ is concave on $(0,1)$.
  Moreover, $f(r)$ is point-symmetric about $r=1/2$, i.e., $f(1-r)=-f(r)$ for all $r\in[0,1]$.
\end{asm}
Systems of the type \eqref{eq:eqs} are widely used as models for spinodal decomposition.
Particularly, the choice \eqref{eq:degenes} of $f$ below describes the demixing of two polymers,
see e.g. \cite{deGennes80, EPM97, OE97}.

An $f$ satisfying Assumption \ref{asm} is \emph{singular} in the sense that
it has infinite slope at the boundary of $[0,1]$.
It is this behaviour which makes the analysis of the problem at hand significantly more challenging
than the corresponding Cahn-Hilliard problem with $f(r)=r-\frac12$
that the authors have considered recently with Nabet \cite{CMN18}.
In the current paper, the example of primary interest is
\begin{align}
  \label{eq:degenes}
  f(r)=\arcsin(2r-1),
  \quad\text{with}\quad
  \frac1{f'(r)^2}= r(1-r).
\end{align}
An alternative admissible choice for $f$ is $f(r)=r^\gamma-(1-r)^\gamma$ with $\frac12\le\gamma<1$.
Note that these functions interpolate between
the linear function $f(r)=2r-1$ at $\gamma\uparrow 1$, corresponding to the Cahn-Hilliard model,
and a function with square-root singularities like in \eqref{eq:degenes} at $\gamma=\frac12$.

The role of $f$ is best understood as follows:
there is a dissipated free energy functional for \eqref{eq:eqs}, which is given by
\begin{align}
  \label{eq:altnrg}
  \altnrg(c_1,c_2) = \frac14\int_\Omega \big( |\nabla f(c_1)|^2+|\nabla f(c_2)|^2 + 2\chi c_1c_2\big)\dd x.
\end{align}
Assumption \ref{asm} guarantees that the gradient parts, i.e.,
\begin{align*}
  c_i\mapsto \int_\Omega |\nabla f(c_i)|^2\dd x,
\end{align*}
are convex functionals.
Consequently, $\altnrg$ is of the form ``convex plus smooth''.
With the choice \eqref{eq:degenes}, $\altnrg$ is referred to as Flory-Huggins-deGennes-energy.

We remark that thermal agitation effects can be incorporated into the model
by augmenting the energy~\eqref{eq:altnrg} with the mixing entropy 
\[
\theta \int_\Omega \big( c_1 \log c_1 + c_2 \log c_2 \big) \dd x, \qquad \theta \geq 0. 
\]
Here we are concerned solely with the so-called deep-quench limit $\theta=0$,
which is analytically the most challenging case.
Indeed, thermal effects introduce additional diffusion to the problem which provide more regularity.

\subsection{Local versus non-local dynamics}
In dimensions $d>1$, there is a subtle difference between the ``non-local'' model under consideration here
and its ``local'' reduction in the sense of de Gennes \cite{deGennes80}.
That difference, and its consequences on the long time asymptotics of solutions, have been discussed in detail in \cite{OE97}.
For the local model, one strengthens the constraint \eqref{eq:eqcons} by requiring
annihilation of the fluxes of $c_1$ and $c_2$ (and not only the divergences of these fluxes),
i.e.,
\begin{align}
  \label{eq:CH}
  m_1c_1\nabla\mu_1 + m_2c_2\nabla\mu_2 = 0.
\end{align}
This condition is stronger than the original constraint \eqref{eq:eqcons} in the sense that
the system consisting of \eqref{eq:eqcont1}, \eqref{eq:eqcont2}, \eqref{eq:eqlagr}, and \eqref{eq:CH} 
propagates \eqref{eq:eqcons} in time.
Moreover, \eqref{eq:CH} allows to eliminate $\mu_2$ from \eqref{eq:eqlagr},
and the system then becomes equivalent to one single evolution equation of fourth order for $c_1$;
in the case $m_1=m_2=1$, it reads
\begin{align}
  \label{eq:oCH}
  \partial_tc_1 = -\dv\big(c_1(1-c_1)\,\nabla\big[f'(c_1)\Delta f(c_1) + \chi (c_1-\smallhalf)\big]\big). 
\end{align}
There seems to be no way to reduce the original system \eqref{eq:eqs}
to a single differential equation in a similar fashion.
The reduction that comes closest to \eqref{eq:oCH} --- still in the case $m_1=m_2=1$ ---
is the following non-local equation, taken from \cite{OE97},
\begin{align}
  \label{eq:reduced}
  \partial_tc_1 = -\dv\big(c_1\mathbf{P}\big\{(1-c_1)\,\nabla\big[f'(c_1)\Delta f(c_1)+ \chi (c_1-\smallhalf)\big]\big\}\big),
\end{align}
in which $\mathbf{P}$ is the Helmholtz projection onto the gradient vector fields.
More explicitly, one combines \eqref{eq:eqcont1} with the following elliptic equation for $\mu_1$:
\begin{align}
  \label{eq:reduced2}
  -\Delta\mu_1 = \dv\big( (1-c_1)\nabla\big[f'(c_1)\Delta f(c_1) + \chi (c_1-\smallhalf)\big]\big),
\end{align}
which is easily derived by adding \eqref{eq:eqcont1} and \eqref{eq:eqcont2},
and using that $\partial_t(c_1+c_2)=0$ because of \eqref{eq:eqcons}.
Despite all the advantages that the reduced equation \eqref{eq:reduced} might have,
the original two-component formulation \eqref{eq:eqs} is the significant one for our existence analysis.

The less restrictive constraint~\eqref{eq:eqcons} provides more flexibility for the fluxes than~\eqref{eq:CH}.
This effect is measurable on the level of energy decay,
which is significantly faster in the non-local model~\eqref{eq:reduced} than in the local model~\eqref{eq:oCH}. 
Numerical evidence of this fact has been presented in~\cite{CN_FVCA, CMN18} in the Cahn-Hilliard case.
On the theoretical side, 
the dynamics of \eqref{eq:oCH} and of \eqref{eq:reduced} have been compared in \cite{OE97} in the sharp interface limit:
this is where $\chi$ is large and the considered time scale is proportional to $\chi$.
Then the values of the solution $c_1$ are concentrated around zero and one,
and the interfaces in between these pure phases become sharper the larger $\chi$ is.
It turns out that the long-time asymptotics of the interfaces in \eqref{eq:oCH} and in \eqref{eq:reduced} are different:
while \eqref{eq:oCH} is asymptotically equivalent to (the slower) surface diffusion,
\eqref{eq:reduced} leads to (the faster) Hele-Shaw flow.
We refer to~\cite{JKM_arXiv} for a recent mathematical study of the interface dynamics
inside the framework of optimal mass transport.

\subsection{Gradient flow structure}\label{ssct:GF}
Similarly as in our recent paper \cite{CMN18},
we take the interpretation of \eqref{eq:eqs} as a metric gradient flow as starting point for the existence analysis.
More specifially,
we use the gradient flow structure to construct time-discrete approximations of the true solution $\bc$
by means of the minimizing movement scheme,
derive a priori estimates on the approximation  by variational methods,
and finally pass to the time-continuous limit.
We emphasize that the interpretation of \eqref{eq:eqs} as gradient flow motivates the aforementioned procedure,
but we are not going to verify that solutions to \eqref{eq:eqs} are curves of steepest descent in a rigorous way.

The potential $\nrg$ of the flow under consideration is essentially the system's free energy $\altnrg$ from \eqref{eq:altnrg},
however, modified such that the volume constraint \eqref{eq:eqcons} is built in:
\begin{align}
  \label{eq:nrg}
  \nrg(\bc) = \nrgone(c_1)+\devil_{c_1+c_2\equiv1}(\bc),
  \quad
  \nrgone(c_1) = \frac12\int_\Omega |\nabla f(c_1)|^2\dd x + \frac\chi2\int_\Omega c_1(1-c_1)\dd x.
\end{align}
Above, $\devil_{c_1+c_2\equiv1}$ denotes the indicator function that is zero if the constraint $c_1+c_2\equiv1$ is satisfied,
and is $+\infty$ otherwise.
$\nrg$'s ``gradient'' is calculated with respect to a metric $\dst$
that combines the squared $L^2$-Wasserstein distances of the components $c_1$ and $c_2$.
More specifically, on the space
\begin{align}
  \label{eq:spc}
  \spc:=\left\{\bc:\Omega\to[0,1]^2\ \middle|
  \ \fint_\Omega c_1\dd x = \rho_1, \; \fint_\Omega c_2\dd x = \rho_2\right\},
  \quad\text{with}\quad
  \rho_1=\fint_\Omega c_1^0\dd x=1-\rho_2,
\end{align}
we introduce the metric $\dst$ by (see Section \ref{sct:prelim} below for the definition of $\wass$)
\begin{align}
  \label{eq:dst}
  \dst\big(\hat\bc,\check\bc\big)^2
  = \frac{\wass(\hat c_1,\check c_1)^2}{m_1} + \frac{\wass(\hat c_2,\check c_2)^2}{m_2}.
\end{align}
In the eyes of the metric $\dst$, the two components of $\bc $ are independent,
and the constraint $c_1+c_2\equiv1$ is enforced only by means of the energy.
This way, the metric $\dst$ inherits all of the established properties of the $L^2$-Wasserstein distance.
In comparision, to the best of our knowledge, very little is known about the metric
that would result by including the constraint already in its definition; see, however, \cite{BBG04}.

\subsection{Estimates}
There are three essential a priori estimates that play a role in our existence proof for \eqref{eq:eqs}.
The first two are consequences of the gradient flow structure outlined above:
first, the energy is non-increasing in time,
and in particular, $\nrg(\bc(t))\le\nrg(\bc^0)$ for each $t\ge0$.
This ensures validity of the constraint~\eqref{eq:eqcons},
and provides a priori estimates of $c_i$ and $f(c_i)$ in $L^\infty(0,T;H^1(\Omega))$.
Second, the curve $\bc$ is $L^2$-absolutely continuous in time with respect to $\dst$,
that is, both components $c_i$ are absolutely continuous in $\wass$.
That means that the kinetic energy densities $\frac{m_i}2c_i|\nabla\mu_i|^2$
--- see the continuity equations \eqref{eq:eqcont1}\&\eqref{eq:eqcont2} ---
are integrable in space and time.
This provides a priori estimate on $\sqrt{c_i}\nabla\mu_i$ in $L^2(\Omega_T)$.

The third estimate is related to the dissipation of an auxiliary functional,
namely the entropy:
\begin{align}
  \label{eq:ent}
  \ent(\bc) = \frac{\tilde\ent(c_1)}{m_1} + \frac{\tilde\ent(c_2)}{m_2},
  \quad\text{where}\quad
  \tilde\ent(c_i)=\int_\Omega c_i(\log c_i-1)+1\dd x.
\end{align}
Indeed, it follows from a formal calculation given below in \eqref{eq:entdiss} 
that $\ent$'s dissipation can be estimated in the form
\begin{align}
  \label{eq:formalcrucial}
  -\frac{\dn}{\dd t}\ent(\bc) \ge \frac1{2d}\int_\Omega \big(\Delta f(c_1)\big)^2\dd x - M,
\end{align}
with some constant $M\ge0$ that is independent of the specific solution $\bc$.
This provides an a priori estimate on $f(c_1)$ in $L^2_\loc(\setR_{>0};H^2(\Omega))$,
which is our main source of compactness.

\subsection{Reformulation of the equations}
\label{ssct:q}
A key element in our existence analysis is a very particular weak formulation of the system \eqref{eq:eqs},
which is taylored to the special nonlinearity under consideration.
In the Cahn-Hilliard case, where $f$ is smooth up to the boundary,
it is possible to define  a proper notion of phase chemical potential $\mu_i$
even when the corresponding phase vanishes, $c_i = 0$, see~\cite{CMN18}.
This approach does not extend easily to the case of singular $f$'s considered here.
Our ansatz is to substitute the bare potentials $\mu_1$ and $\mu_2$,
which are difficult to analyze, by auxiliary quantites $q_1$ and $q_2$ given in \eqref{eq:q1} below.

Some notation is needed:
by Assumption \ref{asm} on $f$, there exists a continuous $\omega:[0,1]\to\setR$ with $\omega(0) = 0$
that is smooth and positive on $(0,1]$ such that
\begin{align}
  \label{eq:8}
  \frac1{f'(r)}=\omega(r)\omega(1-r) \quad\text{for $0<r<1$}.
\end{align}
For notational convenience, we further introduce the continuous function $\alpha:[0,1]\to\setR$
with $\alpha(0)=0$ and $\alpha(r)=r/\omega(r)$ for $r\in(0,1]$;
continuity at $r=0$ is a consequence
of the assumed concavity of $r\mapsto\frac1{f'(r)^2}=\omega(r)^2\omega(1-r)^2$.
For $f$ from \eqref{eq:degenes}, one may choose $\omega(r)=\sqrt{r}$,
and then finds that $\alpha(r)=\sqrt r$ as well.

The auxiliary quantities that replace $\mu_1$ and $\mu_2$ are
\begin{align}
  \label{eq:q1}
  q_1 = \omega(c_1)\,\mu_1,
  \quad
  q_2 = \omega(c_2)\,\mu_2.
\end{align}
The $q_i$ are much better behaved than the $\mu_i$,
since they vanish by definition when $c_i$ does since $\omega(0) = 0$.
Accordingly, the continuity equation \eqref{eq:eqcont1} is interpreted in the following way:
\begin{equation}
  \label{eq:2weak}
  \begin{split}
    \partial_tc_1 &= \dv\left(m_1c_1\nabla\left[\frac{q_1}{\omega(c_1)}\right]\right)
    = m_1\dv\left(\nabla\left[c_1\frac{q_1}{\omega(c_1)}\right]-\nabla c_1\,\frac{q_1}{\omega(c_1)}\right) \\
    &= m_1\dv\big(\nabla[\alpha(c_1)q_1]-\omega(c_2)\nabla f(c_1)\,q_1\big),
  \end{split}
\end{equation}
and similarly for \eqref{eq:eqcont2}.
Concerning the constitutive equation \eqref{eq:eqlagr}:
after multiplication by $1/f'(c_1)$, it can be reformulated in in terms of the $q_i$ as
\begin{align}
  \label{eq:wkconstr}
  \omega(c_1)q_2-\omega(c_2)q_1 = \nonl[c_1]:=\Delta f(c_1) + \chi\omega(c_1)\omega(c_2)\left(c_1-\smallhalf\right),
\end{align}
which makes perfectly sense in view of the $L^2(\Omega_T)$-regularity of $\Delta f(c_1)$. 

The significance of the formulation \eqref{eq:2weak} is that the right-hand side
can be interpreted in the sense of distributions as soon the product $q_1\nabla f(c_1)$ is well-defined.
Since $f(c_1)\in L^2(0,T;H^2(\Omega))\cap L^\infty(0,T;H^1(\Omega))$ thanks to the a priori estimates,
we have $\nabla f(c_1)\in L^3(\Omega_T)$ by interpolation (recall that $d\le3$),
and so it is sufficient that $q_1\in L^{3/2}(\Omega_T)$.
That latter is deduced by means of the representation 
\begin{equation}
  \label{eq:q1alt}
  q_1=\omega(c_1)\bar\mu + \alpha(c_2)\nonl[c_1],
\end{equation}
in which $\bar\mu = c_1\mu_1+c_2\mu_2 = \alpha(c_1)q_1+\alpha(c_2)q_2$ is an average chemical potential. 
The quantity $\nonl[c_1]$ is bounded in $L^2(\Omega_T)$ thanks to the main a priori estimate;
a bound on $\bar\mu$ is obtained from the following representation of $\bar\mu$'s gradient:
\begin{align}
  \nabla\bar\mu
  = c_1\nabla\mu_1+c_2\nabla\mu_2 + \nabla c_1(\mu_1-\mu_2) 
  \label{eq:nablabarmu}
  = \sqrt{c_1}\big(\sqrt{c_1}\nabla\mu_1\big) + \sqrt{c_2}\big(\sqrt{c_2}\nabla\mu_2\big) + \nabla f(c_1)\,\nonl[c_1],
\end{align}
in which the first two terms are controlled thanks to the $L^2(\Omega_T)$-bound on $\sqrt{c_i}\nabla\mu_i$,
and the last term is controlled by a combination of the $L^\infty(0,T;H^1(\Omega))$-bound on $f(c_1)$
and the $L^2(\Omega_T)$-bound on $\Delta f(c_1)$.
This provides an estimate of $\bar\mu$ in $L^2(0,T;W^{1,1}(\Omega))\hookrightarrow L^{3/2}(\Omega_T)$,
and thus also the desired bound on $q_i$ via~\eqref{eq:q1alt}.


\subsection{Main result}
In the following, $C^\infty_{c,n}(\setR_{>0}\times\Omega)$ denotes the space
of all test functions $\xi\in C^\infty(\setR_{\ge0}\times\overline\Omega)$
such that $\xi(t,\cdot)\equiv0$ for all $t\ge0$ outside of some compact time interval $I\subset\setR_{>0}$,
and for which $\xi(t;\cdot)$ satisfies homogeneous Neumann boundary conditions at each $t>0$.

Our main result is the following.
\begin{thm}
  \label{thm:main}
  Let initial data $\bc^0=(c_1^0,c_2^0)$
  with $c_1^0+c_2^0\equiv1$ and $f(c_1^0),f(c_2^0)\in H^1(\Omega)$ be given.
  Then there exists $\bc=(c_1,c_2):\setR_{\ge0}\times\Omega\to[0,1]^2$
  with the following properties:
  \begin{itemize}
  \item \emph{regularity in time:}
    $c_1,c_2$ are H\"older continuous with respect to time as a map into $L^2(\Omega)$.
  \item \emph{regularity in space:}
    $c_1,c_2,f(c_1),\,f(c_2)\in L^\infty(\setR_{\ge0};H^1(\Omega))$
    and $f(c_1),\,f(c_2)\in L^2_\loc(\setR_{\ge0};H^2(\Omega))$
  \item \emph{boundary conditions:}
    $c_1(t),c_2(t)$ satisfy the homogenous Neumann conditions \eqref{eq:neumann} at a.e.\ $t\ge0$
  \item \emph{initial conditions:}
    $c_1(0)=c_1^0$, $c_2(0)=c_2^0$.
  \end{itemize}
  $\bc$ is accompanied by $\bq=(q_1,q_2):\setR_{\ge0}\times\Omega\to\setR^2$
  with $q_1,q_2\in L^{3/2}(\Omega_T)$ for each $T>0$,
  such that the system \eqref{eq:eqs} is satisfied in the following sense:
  \begin{subequations}
    \label{eq:wf}
    \begin{align}
      \label{eq:wfcont}
      &0=\int_0^\infty\int_\Omega \left[\partial_t\xi\,c_i
        +m_iq_i\big(\alpha(c_i)\, \Delta\xi
        +\omega(1-c_i)\,\nabla f(c_i)\cdot\nabla\xi\big)
        \right] \dd x \dd t
      \\
      \nonumber
      &\qquad \text{for $i=1,2$ and all test functions $\xi\in C^\infty_{c,n}(\setR_{>0}\times\Omega)$}, \\ 
      &1=c_1+c_2 \label{eq:c1+c2=1_cont}
        \qquad \text{a.e. on $\setR_{\ge0}\times\Omega$},\\
      &\omega(c_1)\,q_2 - \omega(c_2)\,q_1 = \Delta f(c_1) + \frac\chi2 \textstyle{\left(c_1-c_2\right)}\omega(c_1)\omega(c_2)
        \qquad \text{a.e. on $\setR_{\ge0}\times\Omega$}.
        \label{eq:diff_pot}
    \end{align}
  \end{subequations}
\end{thm}
Notice that the no-flux boundary conditions \eqref{eq:noflux} are encoded
in the weak form \eqref{eq:wfcont} of the continuity equations \eqref{eq:eqcont1}\&\eqref{eq:eqcont2}:
since the test function $\xi$ is only supposed have vanishing normal derivative,
but still may attain arbitrary values on $\partial\Omega$,
a formal integration by parts in \eqref{eq:wfcont} produces a weak form of \eqref{eq:noflux}.

\subsection{Plan of the paper}
In Section \ref{sct:prelim} below, we give a very brief summary of the relevant results 
from the theory of optimal transportation that are needed in our proof of Theorem~\ref{thm:main}.
In Section \ref{sct:mm}, we describe the construction of the time-discrete approximate solutions,
and we derive a priori estimates in Sections~\ref{sct:apriori} and~\ref{sct:bpriori} 
on the approximate volume fractions $\bc$ and phase potentials $\bq$ respectively.
Finally, in Section \ref{sct:convergence}, we pass to the time-continuous limit,
obtaining a weak solution in the sense of Theorem~\ref{thm:main}.

\subsection{Notation}
When we write in the following that some constant depends \emph{only on the parameters of the problem},
then we mean that this constant can in principle be expressed in terms of
the factor $\chi$, the mobilities $m_1$, $m_2$, the averages $\rho_1$, $\rho_2$ from \eqref{eq:spc},
properties of the function $f$, and geometric properties of the domain $\Omega$.

\section{Preliminaries from the theory of optimal transportation}
\label{sct:prelim}
In the section, we briefly recall three alternative definitions of the $L^2$-Wasserstein distance $\wass$;
in the proof of our main result, we need all three of them.
For more information on the mathematical theory of optimal mass transportation,
we refer to the monographs~\cite{Villani-Topics, Villani09, Santa}.
Below, we assume that $\rho_0,\rho_1:\Omega\to[0,1]$ are two measurable functions of the same total mass,
\begin{align*}
  \int_\Omega\rho_0(x)\dd x = \int_\Omega\rho_1(x)\dd x.
\end{align*}
In this case, the definitions are all equivalent.

\subsection{Monge characterization}
One says that a measurable map $T:\Omega\to\Omega$ pushes $\rho_0$ forward to $\rho_1$,
written as $T\#\rho_0=\rho_1$,
if 
\begin{align*}
  \int_\Omega\Theta(x)\rho_1(x)\dd x = \int_\Omega\Theta\circ T(y)\rho_0(y)\dd y
  \quad \text{for all $\Theta\in C^0(\overline\Omega)$}.
\end{align*}
The Monge characterization of the $L^2$-Wasserstein distance between $\rho_0$ and $\rho_1$
is given by
\begin{align}
  \label{eq:W2}
  \wass(\rho_0,\rho_1)^2 = \inf_{T\#\rho_0=\rho_1}\int_\Omega |T(x)-x|^2\rho_0(x)\dd x,
\end{align}
where the infimum runs over all measurable maps $T:\Omega\to\Omega$ with $T\#\rho_0=\rho_1$.
In the situation at hand, the infimum in \eqref{eq:W2} is actually a minimum.
It is attained by an optimal transport map $\Topt$;
the optimal map is uniquely determined on the support of $\rho_0$.

\subsection{Kantorovich characterization}
A Borel measure $\gamma$ on the product space $\Omega\times\Omega$ is called
a transport plan from $\rho_0$ to $\rho_1$ if the latter are the marginals of $\gamma$,
i.e.,
\begin{align*}
  \int_{\Omega\times\Omega}\varphi(x)\dd\gamma(x,y) = \int_\Omega \varphi(x)\rho_0(x)\dd x,
  \quad
  \int_{\Omega\times\Omega}\psi(y)\dd\gamma(x,y) = \int_\Omega \psi(y)\rho_1(y)\dd y,
\end{align*}
for all $\varphi,\psi\in C^0(\Omega)$.
The set of all such transport plans is denoted by $\Gamma(\rho_0,\rho_1)$.
The Kantorovich characterization of $\wass$ amounts to
\begin{align*}
  \wass(\rho_0,\rho_1)^2 = \inf_{\gamma\in\Gamma(\rho_0,\rho_1)}\int_{\Omega\times\Omega}|x-y|^2\dd\gamma(x,y),
\end{align*}
and the infimum is attained by some optimal plan $\gamma_\opt$.
In the situation at hand, $\gamma_\opt$ is unique.
Moreover, it is concentrated on a graph: 
$\gamma_\opt$'s support is contained in $\{(x,\Topt(x))|x\in\Omega\}\subset \Omega\times\Omega$,
where $\Topt$ is an optimal map from the Monge characterization.

\subsection{Dual characterization}
\label{sct:W2dual}
The dual characterization of the Wasserstein distance is given by
\begin{align}
  \label{eq:Wdual}
  \frac12 \wass(\rho_0,\rho_1)^2 = \sup_{\varphi(x)+\psi(y)\le\frac12|x-y|^2}
  \left(\int_\Omega\varphi(x)\rho_0(x)\dd x + \int_\Omega\psi(y)\rho_1(y)\dd y \right),
\end{align}
where the supremum runs over all potentials $\phi,\psi\in C^0(\Omega)$ satisfying $\varphi(x) + \psi(y) \le \frac12|x-y|^2$.
The infimum is attained by a pair of globally Lipschitz functions $(\varphi_\opt,\psi_\opt)$,
which are referred to as Kantorovich potentials.
The potentials are related to the optimal Monge map $\Topt$ via $\Topt(x)=x-\nabla\varphi_\opt(x)$.

There are always infinitely many pairs of Kantorovich potentials,
since the value of the function and the constraint are invariant under the exchange of a global constant,
i.e., $\varphi\leadsto\varphi+C$ and $\psi\leadsto\psi-C$ for any $C\in\setR$.
On the other hand, if at least one of the two densities $\rho_0$ and $\rho_1$ has full support,
then this global constant is the only degree of non-uniqueness.

\section{Time-discrete approximation via minimizing movement scheme}
\label{sct:mm}
As explained in Section~\ref{ssct:GF},
the problem~\eqref{eq:eqs}--\eqref{eq:ic} can be interpreted as the gradient flow
of the singular energy $\nrg$ with respect to the metric $\dst$ on the space $\spc$.
In view of that structure, a natural approach to construction of solutions to \eqref{eq:eqs}
is the time-discrete approximation by means of the minimizing movement scheme.
This approach has been proven extremely robust,
and has been applied for existence proofs in linear and nonlinear Fokker-Planck equations \cite{JKO98,AGS08},
non-local aggregation-diffusion equations \cite{BCC,CDFLS11,KMX17,ZM15},
doubly non-linear and flux-limited equations \cite{Agueh05,McP},
fourth order quantum and lubrication equations \cite{GST,MMS09,LM14,LMS12},
multi-phase flows \cite{LMold,CGM17,CMN18} and many more settings.

In addition to approximations of the volume fractions $c_1$ and $c_2$,
we also need to construct approximations of the auxiliary quantities $q_1$ and $q_2$.
These will be obtained from the Kantorovich potentials
for the optimal transport of the volume fractions between time steps.
In order to ensure that these potentials are well-defined (up to a global additive constant),
we regularize the minimizing movement scheme
by modifying the volume fractions in the previous time step such that both have full support.
This removes the ambiguity in the definition of the Kantorovich potentials, as 
explained in Section~\ref{sct:W2dual}. 

Throughout this section, let two parameters be fixed:
a time step size $\tau>0$, and a positivity regularization $\delta>0$.
We assume that $\tau$ and $\delta$ are related as follows:
\begin{align}
  \label{eq:taudelta}
  \delta\le\tau^2.
\end{align}
Recall the definitions of the energy functional $\nrg$ from \eqref{eq:nrg}
and of the metric $\dst$ from \eqref{eq:dst} on the space $\spc$.
Recall further the definition of the averages $\rho_1$ and $\rho_2$ in \eqref{eq:spc},
and introduce the regularization $[\bc]_\delta=([c_1]_\delta,[c_2]_\delta)$ of a $\bc=(c_1,c_2)\in\spc$
by
\begin{equation}\label{eq:regul}
  [c_i]_\delta = \delta\rho_i + (1-\delta)c_i.
  \end{equation}
With these notations at hand, define for given $\bar\bc\in\spc$ a variational functional in $\bc\in\spc$ by
\begin{align}
  \label{eq:1}
  \nrg_{\tau,\delta}(\bc;\bar\bc) 
  = \frac1{2\tau}\dst\big(\bc,[\bar\bc]_\delta\big)^2 + \nrg(\bc).
\end{align}
At each instance of discretized time $t=n\tau$,
an approximation $(\bc^n,\bq^n)$ of $(\bc(t),\bq(t))$ is constructed as follows.
Starting from the given initial condition $\bc^0$,
each $\bc^n$ is inductively chosen as a global minimizer of $\nrg_{\tau,\delta}(\cdot;\bc^{n-1})$,
i.e.,
\begin{equation}\label{eq:JKO}
  \bc^n \in \operatorname*{argmin}_{\bc\, \in\, \spc} \nrg_{\tau,\delta}(\bc;\bc^{n-1}).
\end{equation}
Solvability of that minimization problem is shown in Lemma \ref{lem:existence} below.

The accompanying auxiliary quantities $q_1^n$ and $q_2^n$ are obtained as follows.
First, let $(\varphi_1^n,\psi_1^n)$ and $(\varphi_2^n,\psi_2^n)$ be two pairs of Kantorovich potentials
for the respective optimal transport of $[c_i^{n-1}]_\delta$ to $c_i^n$;
since $[c_i^{n-1}]_\delta\ge\delta\rho_i$ on $\Omega$, these pairs are unique up to addition of global constants.
These constants are normalized by requiring
\begin{align}
  \label{eq:Knormal}
  \int_\Omega \left[\frac{c_1^n\psi_1^n}{m_1}+\frac{c_2^n\psi_2^n}{m_2}\right]\dd x = 0,
  \quad
  \int_\Omega \left[\psi_2^n-\psi_1^n-\chi\left(c_1-\frac12\right)\right]\omega(c_1^n)\omega(c_2^n)\dd x = 0.
\end{align}
From the $\psi_i^n$, define the rescaled pair of potentials $\bmu^n=(\mu_1^n,\mu_2^n)$ via
\begin{align*}
  \mu_1^n := \frac{\psi_1^n}{m_1\tau},
  \quad
  \mu_2^n := \frac{\psi_2^n}{m_2\tau},
\end{align*}
and finally $\bq^n=(q_1^n,q_2^n)$ is given --- as indicated in \eqref{eq:q1} --- by
\begin{align*}
  q_1^n:=\omega(c_1^n)\mu_1^n,
  \quad
  q_2^n:= \omega(c_2^n)\mu_2^n.
\end{align*}
%
\begin{lem}
  \label{lem:existence}
  Given initial data $\bc^0$ as in Theorem \ref{thm:main},
  the minimization problem for $\bc^n$ can be solved inductively,
  leading to infinite sequences $(\bc^n)_{n\in\setN}$ and $(\bq^n)_{n\in\setN}$.
  The $\bc^n$ satisy the constraint
  \begin{equation}\label{eq:c1+c2=1}
    c_1^n + c_2^n = 1. 
  \end{equation}
\end{lem}
\begin{proof}
  Inductive solvability of the minimization problem follows by the direct methods from the calculus of variations.
  Indeed, it suffices to observe the following about the functional $\nrg_{\tau,\delta}(\cdot;\bc^{n-1})$,
  considered as a map from $\spc$ with the topology of $L^2(\Omega;\setR^2)$ to the extended non-negative real numbers:
  \begin{itemize}
  \item It is bounded below (in fact: is non-negative)
    and is not identically $+\infty$ (e.g., is finite at $\bc^{n-1}$).
  \item It is coercive: 
    if $\tilde\bc^k$ is a sequence in $\spc$ such that $\nrg_{\tau,\delta}(\tilde\bc^k;\bc^{n-1})$ is bounded,
    then in particular $\int_\Omega|\nabla f(\tilde c_1^k)|^2\dd x$ is bounded,
    i.e., $f(\tilde c_1^k)$ is bounded in $H^1(\Omega)$.
    Rellich's compactness theorem now implies strong convergence of a subsequence $f(\tilde c_1^{k'})$ in $L^2(\Omega)$,
    and thanks to the properties of $f$, also $\tilde c_1^{k'}$ itself converges in $L^2(\Omega)$.
    Finally, since finiteness of $\nrg_{\tau,\delta}(\tilde\bc^k;\bc^{n-1})$ implies that $\tilde c_2^k = 1-\tilde c_1^k$,
    convergence of $\tilde c_2^{k'}$ follows as well.
  \item It is lower semi-continuous.
    To see this, let $\tilde\bc^k$ be a sequence in $\spc$ that converges to $\tilde\bc^*$ in $L^2(\Omega;\setR^2)$.
    Convergence of $\dst(\tilde\bc^k,\bc^{n-1})$ and of $\int_\Omega \tilde c_1^k(1-\tilde c_1^k)\dd x$
    towards their respective limits is immediate.
    On the other hand, it follows by continuity of $f$ that also $f(\tilde c_1^k)$ converges to $f(\tilde c_1^*)$ in $L^2(\Omega)$.
    And so,
    \begin{align*}
      \liminf_{k\to\infty}\int_\Omega |\nabla f(\tilde c_1^k)|^2\dd x\ge \int_\Omega|\nabla f(\tilde c_1^*)|^2\dd x
    \end{align*}
    is a consequence of the lower semi-continuity of the $H^1(\Omega)$-norm on $L^2(\Omega)$.
  \end{itemize}
  The relation~\eqref{eq:c1+c2=1} holds since each minimizer $\bc^n$ has a finite energy.
\end{proof}

\section{A priori estimates on the volume fractions}\label{sct:apriori}
The ultimate goal is to obtain solutions $\bc$ and $\bq$ of the weak formulation \eqref{eq:wf}
as appropriate limits of the time-discrete quantities $\bc^n$ and $\bq^n$
for $\tau\downarrow0$ and $\delta\downarrow0$.
In this and the next section,
we establish the a priori estimates that eventually provide sufficient compactness for performing the limit.
As indicated in the introduction, there are three essential estimates:
the first two, given in Lemma \ref{lem:classical} right below, follow almost immediately from the gradient flow structure.
These two estimates are sufficient to conclude the weak convergence of the volume fractions.
The third estimate, given in Lemma \ref{lem:apriori},
follows from the control \eqref{eq:ent} on the production rate of the entropy $\ent$.
It provides strong convergence of the volume fractions and indirectly
--- see Section \label{sct:bpriori} below ---
also weak convergence of the auxiliary functions.
\begin{lem}
  \label{lem:classical}
  There is a constant $K$, only depending on the parameters of the problem, such that for all $N=1,2,\ldots$ 
  \begin{align}
    \label{eq:classical}
    \nrg(\bc^N)+\frac\tau2\sum_{n=1}^N\left(\frac{\dst(\bc^n,[\bc^{n-1}]_\delta)}\tau\right)^2
    \le \nrg(\bc^0) + \frac{K}2N\tau. 
  \end{align}
  Consequently, for all indices $n\le N$ and $\underline n<\overline n\le N$,
  \begin{align}
    \label{eq:totalbound}
    \|\nabla f(c_1^n)\|_{L^2}^2
    &\le 2\nrg(\bc^0) + KN\tau, \qquad \text{for all $n=1,2,\ldots,N$}, \\
    \label{eq:holderW2}
    \dst(\bc^{\overline n},\bc^{\underline n})
    &\le 2\left(\nrg(\bc^0)+KN\tau\right)^{\frac12}\big(\tau(\overline n-\underline n)\big)^{\frac12}
      \qquad \text{for $0\le\underline n<\overline n\le N$}, \\
    \label{eq:holderL2}
    \|\bc^{\overline n} - \bc^{\underline n}\|_{L^2}
    &\le 2\sqrt[4]{m_1}\big(\nrg(\bc^0)+KN\tau\big)^{\frac12} \big(\tau(\overline n-\underline n)\big)^{\frac14}.
  \end{align}
\end{lem}
\begin{proof}
  By definition of $\bc^n$ as a minimizer, $\nrg_{\tau,\delta}(\bc^n;\bc^{n-1})\le\nrg_{\tau,\delta}(\bc^{n-1};\bc^{n-1})$,
  which amounts to
  \begin{align}
    \label{eq:preclassical}
    \nrg(\bc^n)+\frac\tau2\left(\frac{\dst(\bc^n;[\bc^{n-1}]_\delta)}\tau\right)^2
    \le \nrg(\bc^{n-1}) + \frac1{2\tau}\dst(\bc^{n-1},[\bc^{n-1}]_\delta)^2.
  \end{align}
  The last term is bounded by $K\delta/(2\tau)\le K\tau/2$ thanks to Lemma \ref{lem:mollify} from the appendix,
  and to our assumption $\delta\le\tau^2$ from \eqref{eq:taudelta}.
  Summation of \eqref{eq:preclassical} from $n=1$ to $n=N$ yields \eqref{eq:classical},
  and \eqref{eq:totalbound} is an immediate consequence from the definition of $\nrg$.
  To conclude \eqref{eq:holderW2} from here, we use the triangle inequality for $\dst$
  --- which is inherited from $\wass$ ---
  and H\"older's inequality for sums,
  \begin{align*}
    \dst(\bc^{\overline n},\bc^{\underline n})
    \le \sum_{n=\underline n+1}^{\overline n}\dst(\bc^n,\bc^{n-1})
    \le \left(\tau\sum_{n=\underline n+1}^{\overline n}\left(\frac{\dst(\bc^n,\bc^{n-1})}\tau\right)^2\right)^{\frac12}\big(\tau(\overline n-\underline n)\big)^{\frac12}.
  \end{align*}
  The expression inside the first pair of brackets is now estimated with the help of \eqref{eq:classical} above,
  and another application of Lemma \ref{lem:mollify}:
  \begin{align*}
    \tau\sum_{n=\underline n+1}^{\overline n}\left(\frac{\dst(\bc^n,\bc^{n-1})}\tau\right)^2
    \le \tau\sum_{n=\underline n+1}^{\overline n}\left[2\left(\frac{\dst(\bc^n,[\bc^{n-1}]_\delta)}\tau\right)^2+\frac{2K\delta}{\tau^2}\right]
    \le 4\left[\nrg(\bc^0) + \frac{KT}2\right] + \frac{2K\delta}\tau N.
  \end{align*}
  Substitution of this estimate above and recalling \eqref{eq:taudelta} produces \eqref{eq:holderW2}.
  Estimate \eqref{eq:holderL2} emerges as a consequence of \eqref{eq:totalbound} and \eqref{eq:holderW2}
  via Lemma \ref{lem:fromW2toL2} from the appendix.
\end{proof}
The bound \eqref{eq:classical} can be formulated as a weighted $H^1$-estimate on the Kantorovich potentials.
\begin{cor}
  At each $n=1,2,\ldots$, we have that
  \begin{align}
    \label{eq:push}
    [c_1^{n-1}]_\delta = (\id-\nabla\psi^n_1)\#c_1^n,
    \quad\text{and}\quad
    [c_2^{n-1}]_\delta = (\id-\nabla\psi^n_2)\#c_2^n,
  \end{align}
  and therefore, with the same constant $K$ as in Lemma \ref{lem:classical} above,
  for all $N=1,2,\ldots$,
  \begin{align}
    \label{eq:kinetic}
    \tau\sum_{n=1}^N\int_\Omega \left(\frac{c_1^n}{m_1}\left|\frac{\nabla\psi_1^n}\tau\right|^2+\frac{c_2^n}{m_2}\left|\frac{\nabla\psi_2^n}\tau\right|^2\right)\dd x
    \le 2\nrg(\bc^0) + \frac{K\delta N}\tau.
  \end{align}
\end{cor}
\begin{proof}
  The relations \eqref{eq:push} express the property of the Kanotrovich potential $\psi_i^n$
  that $x\mapsto x-\nabla\psi_i^n(x)$ is a transport map from $c_i^n$ to $[c_i^{n-1}]_\delta$.
  In fact, it is the \emph{optimal} transport map, see Section \ref{sct:W2dual},
  and hence \eqref{eq:W2} implies that
  \begin{align*}
    \left(\frac{\dst\big(\bc^n,[\bc^{n-1}]_\delta\big)}\tau\right)^2
    = \frac{\wass(c_1^n,[c_1^{n-1}]_\delta)^2}{m_1\tau^2} + \frac{\wass(c_2^n,[c_2^{n-1}]_\delta)^2}{m_2\tau^2} 
    = \int_\Omega \left(\frac{c_1^n}{m_1}\left|\frac{\nabla\psi_1^n}\tau\right|^2
    +\frac{c_2^n}{m_2}\left|\frac{\nabla\psi_2^n}\tau\right|^2\right)\dd x.
  \end{align*}
  By non-negativity of $\nrg$, the desired estimate \eqref{eq:kinetic} is now implied by \eqref{eq:classical}.
\end{proof}
The third a priori estimate below is more specific to the system \eqref{eq:eqs},
and is also more difficult to prove.
\begin{lem}
  \label{lem:apriori}
  There is a constant $C$, only depending on the parameters of the problem,
  such that for all $N=1,2,\ldots$:
  \begin{align}
    \label{eq:fbound}
    \tau\sum_{n=1}^N\int_\Omega\big\| f(c_1^n)\big\|_{H^2}^2\dd x &\le C(1+N\tau).
  \end{align}
  Moreover, $c_1^n$ and $f(c_1^n)$ satisfy homogeneous Neumann boundary conditions at each $n=1,2,\ldots$ 
\end{lem}
\begin{rmk}
  If $1/(f')^2$ has a bounded derivative 
  --- as is the case for the $f$ from \eqref{eq:degenes} ---
  then one also obtains the analogous estimate as \eqref{eq:fbound} for $c_1$ itself in place of $f(c_1)$.
  Indeed, with $c_1=f^{-1}(f(c_1))$,
  \begin{align*}
    \Delta c_1 = \frac1{f'(c_1)}\Delta f(c_1) - \frac{f''(c_1)}{f'(c_1)^3}|\nabla f(c_1)|^2,
  \end{align*}
  with bounded factors
  \begin{align*}
    \frac1{f'} \quad\text{and}\quad
    -\frac{f''}{(f')^3}=\frac12\left(\frac1{(f')^2}\right)'.
  \end{align*}
  Combining this with the interpolation inequality
  \begin{align*}
    \|\nabla f\|_{L^4}^2\le 3\|f\|_{L^\infty}\|f\|_{H^2},
  \end{align*}
  that is easily derived using integration by parts,
  shows that $\|\Delta c_1\|_{L^2}^2\le C\|f(c_1)\|_{H^2}^2$,
  and therefore, see \eqref{eq:savare} below, also $\|c_1\|_{H^2}^2\le C\|f(c_1)\|_{H^2}^2$.
\end{rmk}
We divide the proof of Lemma \ref{lem:apriori} into two parts:
the first part contains the formal calculations --- for smooth and positive classical solutions to \eqref{eq:eqs} --- that lead to \eqref{eq:formalcrucial},
the second part is the fully rigorous justification of \eqref{eq:fbound} as a time-discrete version of \eqref{eq:formalcrucial},
using the flow interchange technique from \cite{MMS09}.
\begin{proof}[Formal calculation leading to \eqref{eq:formalcrucial}]
  Assume that a smooth and classical solution $\bc$ to \eqref{eq:eqs} with $0<c_1<1$ is given.
  We consider the dissipation of the entropy functional defined in \eqref{eq:ent}.
  We have, thanks to the no-flux and Neumann boundary conditions \eqref{eq:bc},
  \begin{equation}
    \label{eq:entdiss}
    \begin{split}
      -\frac{\dn}{\dd t}\ent(\bc)
      & = -\int_\Omega \left[\frac{\log c_1}{m_1}\,\partial_tc_1+\frac{\log c_2}{m_2}\,\partial_tc_2\right]\dd x \\
      & = \int_\Omega\big[\nabla c_1\cdot\nabla\mu_1 + \nabla c_2\cdot\mu_2]\dd x \\
      & = \int_\Omega \nabla c_1\cdot\nabla[\mu_1-\mu_2]\dd x \\
      &= \int_\Omega f'(c_1)\Delta c_1\Delta f(c_1)\dd x - \chi\int_\Omega|\nabla c_1|^2\dd x.    
    \end{split}
  \end{equation}
  We shall now use various manipulations to obtain a lower bound on
  \begin{align}
    \label{eq:J}
    J:=\int_\Omega f'(c_1)\Delta c_1\Delta f(c_1)\dd x.
  \end{align}
  On the one hand,
  \begin{align}
    \label{eq:Deltaf}
    \Delta f(c_1) = f'(c_1)\Delta c_1 + f''(c_1)|\nabla c_1|^2.
  \end{align}
  And on the other hand,
  thanks to the homogeneous Neumann boundary conditions from \eqref{eq:neumann}
  --- that are inherited from $c_1$ to $f(c_1)$ thanks to $0<c_1<1$ --- 
  and the convexity of $\Omega$,
  we have that (see e.g., \cite[Lemma 5.1]{GST})
  \begin{align}
    \label{eq:savare}
    \int_\Omega \big(\Delta f(c_1)\big)^2\dd x \ge \int_\Omega\|\nabla^2f(c_1)\|^2\dd x,
  \end{align}
  where $\|A\|=\sqrt{\tr(A^TA)}$ is the Frobenius norm of the square matrix $A$.
  Thus, we obtain 
  \begin{align}
    \label{eq:help002}
    J\ge\int_\Omega \left[\|\nabla^2f(c_1)\|^2-\frac{f''(c_1)}{f'(c_1)^2}|\nabla f(c_1)|^2\Delta f(c_1)\right]\dd x.
  \end{align}
  Now introduce $f,g:(0,1)\to\setR$ by
  \begin{align*}
    g(r):=\frac{f''(r)}{f'(r)^2}, 
    \quad
    h(r):= \frac{g'(r)}{f'(r)}, 
  \end{align*}
  and notice that
  \begin{align}
    \label{eq:AB}
    \nabla g(c_1) = h(c_1)\,\nabla f(c_1).
  \end{align}
  In the following, we write shortly $f$, $g$ and $h$ for $f(c_1)$, $g(c_1)$ and $h(c_1)$.
  Thanks again to the homogeneous Neumann boundary conditions and to \eqref{eq:AB},
  the divergence theorem implies that
  \begin{align*}
    0 &= \frac d{d+2}\int_\Omega \dv\big(g|\nabla f|^2\nabla f\big)\dd x \\
    &= \frac d{d+2}\int_\Omega \big[g\Delta f|\nabla f|^2+2g\nabla f\cdot\nabla^2f\nabla f+h|\nabla f|^4\big]\dd x.
  \end{align*}
  Adding the final integral expression to the right-hand side of \eqref{eq:help002} produces
  \begin{align*}
    J\ge
    \int_\Omega \left[\|\nabla^2f\|^2-\frac{2g}{d+2}\Delta f|\nabla f|^2
    +\frac{2dg}{d+2}\nabla f\cdot\nabla^2f\cdot\nabla f+\frac{dh}{d+2}|\nabla f|^4\right]\dd x.    
  \end{align*}
  Next, introduce the matrix-valued function $R:\Omega\to\setR^{d\times d}$ by
  \begin{align*}
    R:=\nabla^2f(c_1)-\frac{\Delta f(c_1)}{d}\eins_d.
  \end{align*}
  Then, using that $\tr\eins_d=d$ and $\tr R=0$,
  we obtain that 
  \begin{align*}
    \|\nabla^2f(c_1)\|^2 = \tr\left[\left(R+\frac{\Delta f(c_1)}{d}\eins_d\right)^2\right]
    =   \|R\|^2 + \frac{\left(\Delta f(c_1)\right)^2}{d}, 
  \end{align*}
  which allows to conclude that
  \begin{align*}
    J\ge
    \frac1d\int_\Omega (\Delta f)^2\dd x
    +\int_\Omega \left[\|R\|^2
    +\frac{2dg}{d+2}\nabla f\cdot R\cdot\nabla f+\frac{dh}{d+2}|\nabla f|^4\right]\dd x.    
  \end{align*}
  The last step is to verify that the expression inside the final integral is pointwise non-negative:
  \begin{align*}
    \|R\|^2 +\frac{2dg}{d+2}\nabla f\cdot R\cdot\nabla f+\frac{dh}{d+2}|\nabla f|^4
    = \left\| R + \frac{dg}{d+2}\nabla f\,\nabla f^T\right\|^2
    + \left[\frac{dh}{d+2}-\left(\frac{dg}{d+2}\right)^2\right]|\nabla f|^4.
  \end{align*}
  The squared norm is trivially non-negative.
  For the coefficient of the term $|\nabla f|^4$ to be non-negative,
  it suffices to have $h\ge g^2$.
  Since
  \begin{align*}
    h-g^2
    = \frac1{f'}\left(\frac{f''}{(f')^2}\right)' - \left(\frac{f''}{(f')^2}\right)^2
    = \frac{f'''}{(f')^3} - 3\frac{(f'')^2}{(f')^4}
    = -\frac12\left(\frac1{(f')^2}\right)'',
  \end{align*}
  the assumed concavity of $1/(f')^2$ is sufficient to guarantee $h\ge g^2$.
  In summary,
  \begin{align}
    \label{eq:J1}
    J\ge\frac1d\int_\Omega\Delta f(c_1)^2\dd x.
  \end{align}
  It remains to estimate the other integral.
  Recall that $f$ is continuous, and that $f'$ is positive with $1/(f')^2$ concave,
  so there is a constant $a>0$ such that $|f|\le a$ and $f'\ge a^{-1}$.
  Since $\Omega$ is bounded,
  and thanks to the Neumann boundary conditions \eqref{eq:neumann},
  \begin{equation}
    \label{eq:chiest}
    \begin{split}
      \chi\int_\Omega|\nabla c_1|^2\dd x
      \le a^2\chi\int_\Omega|\nabla f(c_1)|^2\dd x
      &= -a^2\chi\int_\Omega f(c_1)\Delta f(c_1)\dd x\\
      &\le a^2\chi\big(a^2|\Omega|\big)^{\frac12}\big(dJ\big)^{\frac12}
      \le \frac1{2}J+\frac{d\chi^2a^4}{2}|\Omega|.
    \end{split}
  \end{equation}
  Going back to \eqref{eq:entdiss},
  we arrive at \eqref{eq:formalcrucial}, or more specifically:
  \begin{align}
    \label{eq:algebraprefinal}
    -\frac{\dn}{\dd t}\ent(\bc)
    \ge \frac1{2d}\int_\Omega(\Delta f)^2\dd x - \frac{d\chi^2a^4}{2}|\Omega|.
  \end{align}
  An integration of this inequality in time
  provides
  \begin{align}
    \label{eq:algebrafinal}
    \int_0^T\int_\Omega(\Delta f)^2\dd x 
    \le 2d\big[\ent(\bc^0)-\ent\big(\bc(T)\big)\big]
    + d^2\chi^2a^4|\Omega|T.
  \end{align}
  Notice that the value of the entropy $\ent(\bc)$ is uniformly bounded from above and below
  for all $\bc\in\spc$.
  The estimate \eqref{eq:fbound} under consideration is a time-discrete version of \eqref{eq:algebrafinal},
  using that the integral over $\Delta f(c_1)$ on the left hand side
  yields control on the $H^2$-norm of $f(c_1)$
  by means of another application of \eqref{eq:savare}
  and interpolation with the trivial $L^\infty(L^2)$-bound on $c_1$. 
\end{proof}
\begin{proof}[Making the formal calculations rigorous]
  For each fixed $n$, we show the following time-step version of \eqref{eq:algebrafinal}:
  \begin{align}
    \label{eq:onestep}
    \tau\int_\Omega \big(\Delta f(c_1^n)\big)^2\dd x 
    \le 2d\big[\ent(\bc^{n-1})-\ent(\bc^n)\big] + \tau(d^2\chi^2a^4|\Omega|+K),
  \end{align}
  where $K$ is independent of $\tau$.
  With \eqref{eq:onestep} at hand, the estimate \eqref{eq:fbound} follows by summation over $n=1,2,\ldots,N$.

  The starting point for the derivation of \eqref{eq:onestep} is a particular variation
  of the minimizer  $\bc_\tau^n$ of $\nrg_{\tau,\delta}(\cdot;\bc_\tau^{n-1})$:
  consider the family $\bc^s=(c_1^s,c_2^s)\in\spc$,
  where $c_1^s$ and $c_2^s$ are the time-$s$-solutions to the heat flow
  on $\Omega$ for data $c_1^n$ and $c_2^n$, with homogeneous Neumann boundary conditions:
  \begin{subequations}\label{eq:heat}
    \begin{align}
      \label{eq:heat.eq}
      &\partial_s c_i^s = \Delta c_i^s \quad \text{for}\; (s,x) \in \setR_{>0} \times \Omega,\\
      \label{eq:heat.bc}
      & \nml\cdot\nabla c_i^s = 0 \quad \text{on}\;  \setR_{>0} \times \partial \Omega, \\
      \label{eq:heat.ic}
      & {c_i^s}_{|_{s=0}} = c_i^{n} \quad \text{in}\; \Omega.
    \end{align}
  \end{subequations}
  The pair $\bc^s=(c_1^s,c_2^s)$ has a variety of nice properties that facilitate the further analysis.
  Thanks to the smoothing effect of the heat equation,
  the map $(s,x)\mapsto c_i^s(x)$ is a $C^\infty$-function on $\setR_{>0}\times\Omega$,
  and it satisfies
  both the equation \eqref{eq:heat.eq} and the boundary condition \eqref{eq:heat.bc} in the classical sense.
  Moreover, one has $0<\inf_xc_i^s(x)\le\sup_xc^s_i(x)<1$ for each $s>0$,
  which implies that the map $(s;x)\mapsto f(c_i^s(x))$ inherits the $C^\infty$-smoothness
  as well as the homogeneous Neumann boundary conditions,
  \begin{align}
    \label{eq:heat.f}
    \nml\cdot\nabla f(c_i^s) \quad \text{on}\;  \setR_{>0} \times \partial \Omega.
  \end{align}
  Concerning the attainment of the initial condition \eqref{eq:heat.ic}:
  it follows from $\nrg(\bc^n)<\infty$ that $f(c_i^n)\in H^1(\Omega)$,
  and hence also $c_i^n\in H^1(\Omega)$ in view of Assumption \ref{asm}.
  This implies
  \begin{align}
    \label{eq:heatH1}
    c_i^s\to c_i^n \quad\text{in $H^1(\Omega)$ as $s\downarrow0$}.
  \end{align}
  Note, however, that we cannot conclude $f(c_i^s)\to f(c_i^n)$ in $H^1(\Omega)$ from here
  because of $f'(r)\to+\infty$ for $r\downarrow0$ and for $r\uparrow1$.
  Finally, the incompressibility constraint  is preserved,
  \begin{align}
    \label{eq:heatincompress}
    c_1^s+c_2^s=1.
  \end{align}
  There are many further possibilities for the perturbation $(\bc^s)$ that would share the aforementioned properties.
  Our motivation for the particular choice \eqref{eq:heat} is that
  solutions to the heat equation form a so-called $\mathrm{EVI}_0$-flow
  of the entropy $\tilde\ent$ in the $L^2$-Wasserstein metric \cite[Theorem 11.1.4]{AGS08};
  we emphasize that convexity of $\Omega$ is essential here.
  The $\mathrm{EVI}_0$-property means that $\setR_{\ge0}\ni s\mapsto\wass(c_i^s,[c_i^{n-1}]_\delta)^2$ is absolutely continuous
  --- and in particular differentiable at almost every $s>0$ --- 
  and that its derivative satisfies
  \begin{align}
    \label{eq:fromevi}
    \frac12\limsup_{s\downarrow0}\frac{\dn}{\dd s}\wass(c_i^s,[c_i^{n-1}]_\delta)^2
    \le \tilde\ent([c_i^{n-1}]_\delta)-\tilde\ent(c_i^n).
  \end{align}
  We combine \eqref{eq:fromevi} with the fact
  that $\nrg_{\tau,\delta}(\bc^s;\bc^{n-1})\ge\nrg_{\tau,\delta}(\bc^n;\bc^{n-1})$ by definition of $\bc^n$ as a minimizer.
  The latter can be equivalently formulated as
  \begin{align*}
    \nrg(\bc^n)-\nrg(\bc^s) \le \frac1{2\tau}\big[\dst(\bc^s,[\bc^{n-1}]_\delta)^2-\dst(\bc^n,[\bc^{n-1}]_\delta)^2\big].
  \end{align*}
  Plugging in the definition of $\dst$,
  dividing by $s>0$, and passing to the limit $s\downarrow0$ yields in view of \eqref{eq:fromevi}:
  \begin{align*}
    \limsup_{s\downarrow0}\frac{\nrg(\bc^n)-\nrg(\bc^s) }s
    & \le \frac1{2\tau}\limsup_{s\downarrow0}\frac{\dn}{\dd s}\Big(\dst(\bc^s,[\bc^{n-1}]_\delta)^2\Big) \\
    & \le \frac1{2m_1\tau}\limsup_{s\downarrow0}\frac{\dn}{\dd s}\wass(c_1^s;[c_1^{n-1}]_\delta)^2 
      +\frac1{2m_2\tau}\limsup_{s\downarrow0}\frac{\dn}{\dd s}\wass(c_2^s;[c_2^{n-1}]_\delta)^2 \\
    &\le \frac{\tilde\ent([c_1^{n-1}]_\delta)-\tilde\ent(c_1^n)}{m_1\tau}
      +\frac{\tilde\ent([c_2^{n-1}]_\delta)-\tilde\ent(c_2^n)}{m_2\tau} 
      = \frac{\ent([\bc^{n-1}]_\delta)-\ent(\bc^n)}\tau.
  \end{align*}
  For simplification of the left-hand side above,
  observe that $\nrg(\bc^n)-\nrg(\bc^s)=\nrgone(c_1^n)-\nrgone(c_1^s)$ thanks to \eqref{eq:heatincompress}.
  For further estimation of the right-hand side, 
  we use that $\ent$ is a non-negative convex functional,
  and thus
  \begin{align*}
    \ent([\bc^{n-1}]_\delta) \le (1-\delta)\ent(\bc^{n-1}) + \delta\ent(\brho)
    \le \ent(\bc^{n-1}) + K\delta,
  \end{align*}
  where $K=\ent(\brho)$ depends only on the parameters of the problem
  In summary, we have obtained so far that
  \begin{align}
    \label{eq:intermediate}
    \limsup_{s\downarrow0}\frac{\nrgone(c_1^n)-\nrgone(c_1^s) }s
    \le \frac{\ent(\bc^{n-1})-\ent(\bc^n)}\tau + K\frac\delta\tau.    
  \end{align}
  The remaining step is to derive a lower bound on the expression on the left-hand side in \eqref{eq:intermediate}
  of the same form as the right-hand side in \eqref{eq:algebraprefinal}.
  Ideally, we would like to express the left-hand side of \eqref{eq:intermediate}
  by means of the fundamental theorem of calculus as an average of $-\dn\nrgone(c_1^s)/\dn s$.
  The technical difficulty here is that $\nrgone(c_1^s)\to\nrgone(c_1^n)$ as $s\downarrow0$ might fail;
  note that Assumption \ref{asm} guarantees lower --- but a priori \emph{not upper} --- semi-continuity of $\nrgone$
  with respect to the $H^1$-convergence \eqref{eq:heatH1}.
  To overcome this, introduce for $\eps\in(0,1)$ the following approximations $f_\eps:[0,1]\to\setR$ of $f$:
  \begin{align*}
    f_\eps\left(\frac12+z\right) = (1-\eps)^{-1}f\left(\frac12+(1-\eps)z\right)
    \quad\text{for $-\frac12\le z\le\frac12$}.
  \end{align*}
  Thanks to Assumption \ref{asm}, $f_\eps'$ is positive, $1/(f_\eps')^2$ is concave, and $f_\eps(1-r)=-f_\eps(r)$.
  Moreover,
  since $f'(\frac12+z)$ is non-decreasing for $z>0$ and non-increasing for $z<0$
  thanks to concavity of $1/(f')^2$ and symmetry of $f'(r)$ about $r=\frac12$,
  it follows with $f_\eps'(\frac12+z)=f'(\frac12+(1-\eps)z)$ for all $z\in[-\frac12,\frac12]$ that
  \begin{align}
    \label{eq:fepsf}
    0<f_\eps'(r)\le f'(r) \quad\text{for all $r\in[0,1]$.}
  \end{align}
  Observe further that $f_\eps:[0,1]\to\setR$ is smooth up to the boundary, and in particular, $f_\eps'$ is bounded.
  Therefore, the desired continuity, i.e., $f_\eps(c_1^s)\to f_\eps(c^n_1)$ in $H^1(\Omega)$ as $s\downarrow0$,
  follows directly from \eqref{eq:heatH1}.

  From the smoothness of $c_1^s$ for $s>0$
  it follows in particular that $f_\eps(c_1^s)$ is a smooth curve in $H^1(\Omega)$ for $s>0$
  with $\partial_s f_\eps(c_1^s)=f'_\eps(c_i^s)\Delta c_1^s$.
  Observing further that $f_\eps(c_1^s)$ satisfies homogeneous Neumann boundary conditions since $c_1^s$ does,
  the fundamental theorem of calculus now implies for any $\bar s>0$ that
  \begin{align*}
    \frac12\int_\Omega |\nabla f_\eps(c_1^n)|^2\dd x
    - \frac12\int_\Omega |\nabla f_\eps(c_1^{\bar s})|^2\dd x
    &= - \int_0^{\bar s} \int_\Omega \nabla f_\eps(c_1^s)\cdot\nabla\partial_sf_\eps(c_1^s)\dd x\dd s \\
    &= \int_0^{\bar s}\int_\Omega f'_\eps(c_1^s)\Delta c_1^s\Delta f_\eps(c_1^s)\dd x\dd s.
  \end{align*}
  The integrand for the $s$-integral is of the form $J$ in \eqref{eq:J}.
  Since in the derivation of \eqref{eq:J1}, no property of $f$ other than smoothness, positivity of $f'$, and concavity of $1/(f')^2$ was used,
  the estimate \eqref{eq:J1} also holds with $f_\eps$ in place of $f$, i.e.,
  \begin{align*}
    \int_\Omega f'_\eps(c_1^s)\Delta c_1^s\Delta f_\eps(c_1^s)\dd x
    \ge \frac1d \int_\Omega \big[\Delta f_\eps(c_1^s)\big]^2\dd x
  \end{align*}
  for each $s>0$;
  the technical hypotheses for the derivation of \eqref{eq:J1} ---
  smoothness of $c_1^s$, the bounds $0<c_1^s<1$,
  and the homogeneous Neumann boundary conditions for $f(c_1^s)$ ---
  are guaranteed by the properties of the heat flow.

  Now we pass to the limit $\eps\downarrow0$.
  On the one hand, we can directly estimate
  $\int_\Omega |\nabla f_\eps(c_1^n)|^2\dd x \le \int_\Omega |\nabla f(c_1^n)|^2\dd x$ thanks to \eqref{eq:fepsf}.
  On the other hand, 
  using that $f_\eps(c_1^s)\to f(c_1^s)$ uniformly
  as well as the lower semi-continuity of the $H^1$- and the $H^2$-semi-norms with respect to convergence in measure,
  we finally arrive at
  \begin{align}
    \label{eq:fiA}
    \frac12\int_\Omega |\nabla f(c_1^n)|^2\dd x
    - \frac12\int_\Omega |\nabla f(c_1^{\bar s})|^2\dd x
    \ge \frac1d\int_0^{\bar s}\int_\Omega \big[\Delta f(c_1^s)\big]^2\dd x\dd s.
  \end{align}
  Another --- this time completely straight-forward --- application of the fundamental theorem of calculus provides
  \begin{equation}
    \label{eq:fiB}
    \begin{split}
      &\frac\chi2\int_\Omega c_1^n(1-c_1^n)\dd x
      - \frac\chi2\int_\Omega c_1^s(1-c_1^{\bar s})\dd x 
      = -\frac\chi2\int_0^{\bar s}\int_\Omega (1-2c_1^s)\partial_sc_1^s\dd x\dd s \\
      &= \chi\int_0^{\bar s}\int_\Omega|\nabla c_1^s|^2\dd x\dd s
      \ge -\int_0^{\bar s}\left(\frac1{2d}\int_\Omega \big[\Delta f(c_1^s)\big]^2\dd x
        + \frac{d\chi^2a^4|\Omega|}2\right) \dd s,
    \end{split}
  \end{equation}
  where we have derived the last estimate in analogy to \eqref{eq:chiest},
  with $a>0$ defined there.
  Summation of \eqref{eq:fiA} and \eqref{eq:fiB} yields
  \begin{align}
    \label{eq:1234}
    \int_0^{\bar s}\left(\frac1{2d}\int_\Omega \big[\Delta f(c_1^n)\big]^2\dd x - \frac{d\chi^2a^4|\Omega|}2\right)\dd s
    \le \nrgone(c_1^n)-\nrgone(c_1^{\bar s}).
  \end{align}
  We substitute this estimate into \eqref{eq:intermediate} and obtain,
  using again the lower semi-continuity of the $H^2$-semi-norm,
  \begin{align*}
    \frac1{2d}\int_\Omega \big[\Delta f(c_1^s)\big]^2\dd x - \frac{d\chi^2a^2|\Omega|^4}2
    \le \frac{\ent(\bc^{n-1})-\ent(\bc^n)}\tau + K\frac\delta\tau.
  \end{align*}
  Recalling \eqref{eq:taudelta}, this is \eqref{eq:onestep}.
  
  Concerning the boundary condition:
  the estimate \eqref{eq:fiA} implies in particular that there is a sequence $(s_k)$ of $s_k>0$ with $s_k\downarrow0$
  such that $\Delta f(c_1^{s_k})$ is bounded in $L^2(\Omega)$.
  This implies weak convergence of a further subsequence $f(c_1^{s_k'})$ to $f(c_1^n)$ in $H^2(\Omega)$,
  and this is sufficient to conclude that the normal trace $\nml\cdot\nabla f(c_1^{s_k'})$
  converges weakly in $L^2(\partial\Omega)$ to $\nml\cdot\nabla f(c_1^n)$.
  In particular, $f(c_1^{s_k'})$'s homogeneous Neumann boundary condition is inherited by $f(c_1^n)$,
  and by Assumption \ref{asm}, also $c_1^n$ itself satisfies homogeneous Neumann conditions.
\end{proof}
\begin{cor}
  There is a constant $C$, only depending on the parameters of the problem,
  such that, for all $N=1,2,\ldots$,
  \begin{align}
    \label{eq:F}
    \tau\sum_{n=1}^N\big\|\nonl[c_1^n]\big\|_{L^2}^2\le C(1+N\tau).
  \end{align}
\end{cor}
\begin{proof}
  This follows immediately from \eqref{eq:fbound}, since
  \begin{align*}
    \|\nonl[c_1^n]\|_{L^2} \le \big\|\Delta f(c_1^n)\big\|_{L^2}+\chi\|\omega\|_{C^0}^2\left\|c_1^n-\frac12\right\|_{L^2}
    \le \big\|f(c_1^n)\big\|_{H^2} + M,
  \end{align*}
  where $M$ only depends on the parameters of the problem.
\end{proof}

\section{A priori estimates on the auxiliary potentials}\label{sct:bpriori}
The aim of the current section is to derive --- on the basis of the estimates on $\bc^n$ ---
a priori estimates on the discrete approximation of the auxiliary functions $\bq^n$.
We start by showing that thanks to our construction of $\bc^n$ and $\bq^n$,
the constitutive equation \eqref{eq:wkconstr} holds with $\bc^n$ in place of the true solution $\bc$.
Recall the definition of $\nonl$ given there.
\begin{prp}
  \label{prp:EuLa}
  At each $n\geq 1$,
  \begin{align}
    \label{eq:EuLa}
    \omega(c_1^n)q_2^n-\omega(c_2^n)q_1^n = \nonl[c_1^n].
  \end{align} 
\end{prp}
\begin{proof}
  Thanks to the continuity of $\bc^n$ --- recall that $H^2(\Omega) \subset C(\overline\Omega)$ since $d\leq 3$ --- 
  the set $P=\{x\in\Omega\,|\,0<c_1^n(x)<1\}$ is open.
  Let $\eta\in C^\infty_c(\Omega)$ with support in $P$, and of vanishing mean, i.e., $\int_\Omega\eta(x)\dd x=0$.
  Define $\tilde\bc^h=(\tilde c_1^h,\tilde c_2^h)$ with $\tilde c_1^h:=c_1^n+h\eta$ and $\tilde c_2^h:=c_2^n-h\eta$
  for all $h>0$ sufficiently small such that $\tilde\bc^h\in\spc$.  
  Then $\nrg_{\tau,\delta}(\tilde \bc^h;\bc^{n-1})\ge\nrg_{\tau,\delta}(\bc^n;\bc^{n-1})$ by definition of $\bc^n$ as a global minimizer.

  Recall that the $(\varphi_i^n,\psi_i^n)$ are the (uniquely determined since $[c_i^{n-1}]_\delta >0$) pairs of Kantorovich potentials
  for the optimal transport from $[c_i^{n-1}]_\delta$ to $c_i^n$ normalized by \eqref{eq:Knormal}.
  Analogously, let $(\tilde\varphi_i^h,\tilde\psi_i^h)$ be the (still uniquely determined) pair of potentials
  for the optimal transport from $[c_i^{n-1}]_\delta$ to $\tilde c_i^h$,
  normalized such that $\varphi_i^h(\bar x)=\varphi_i^n(\bar x)$ for all $h>0$ at some arbitrarily chosen $\bar x\in\Omega$.
  By the stability of optimal pairs, see e.g. \cite[Theorem 1.52]{Santa},
  it follows that $\tilde\varphi_i^h\to\varphi_i^n$ and $\tilde\psi_i^h\to\psi_i^n$ uniformly on $\Omega$ as $h\to0$.
  
  Using the dual characterization \eqref{eq:Wdual} of the Wasserstein distance, we obtain:
  \begin{align*}
    &\int_\Omega \big(\tilde\psi_1^hc_1^n + \tilde\varphi_1^h[c_1^{n-1}]_\delta +  \tilde\psi_2^hc_2^n + \tilde\varphi_2^h[c_2^{n-1}]_\delta\big)\dd x
    + \nrgone(c_1^n) \\
    &\le \int_\Omega \big[\psi_1^n c_1^n + \varphi_1^n[c_1^{n-1}]_\delta +  \psi_2^n c_2^n + \varphi_2^n[c_2^{n-1}]_\delta\big]\dd x
      + \nrgone(c_1^n) \\
    &\quad = \nrg_{\tau,\delta}(\bc^n;\bc^{n-1}) \\
    &\quad\quad \le \nrg_{\tau,\delta}(\tilde\bc^h;\bc^{n-1}) \\
    &\quad\quad\quad = \int_\Omega \big(\tilde\psi_1^h\tilde c_1^h + \tilde\varphi_1^h[c_1^{n-1}]_\delta
      +  \tilde\psi_2^h\tilde c_2^h + \tilde\varphi_2^h[c_2^{n-1}]_\delta\big)\dd x
      + \nrgone(\tilde c_1^h).
  \end{align*}
  Subtracting the first line from the ultimate one, and dividing by $h>0$ yields:
  \begin{equation}
    \label{eq:forEL}
    \begin{split}
      0 \le
      &\frac1h\int_\Omega \big(\tilde\psi_1^h (\tilde c_1^h-c_1^n)+ \tilde\psi_2^h(\tilde c_2^h-c_2^n)\big) \dd x
      + \frac{\nrgone(\tilde c_1^h)-\nrgone(c_1^n)}h \\
      &= \int_\Omega (\tilde\psi_1^h-\tilde\psi_2^h)\eta \dd x + \int_\Omega \frac{|\nabla f(\tilde c_1^h)|^2-|\nabla f(c_1^n)|^2}{2h}\dd x
      + \chi\int_\Omega \frac{\tilde c_1^h\tilde c_2^h-c_1^nc_2^n}{2h}\dd x.
    \end{split}
  \end{equation}
  On the one hand, it follows immediately by boundedness of $\eta$ that
  \begin{align}
    \label{eq:chih}
    \int_\Omega \frac{\tilde c_1^h\tilde c_2^h-c_1^nc_2^n}{2h}\dd x
    = \int_\Omega \frac{h(c_2^n-c_1^n)\eta-h^2\eta^2}{2h}\dd x
    \to \frac12\int_\Omega(c_2^n-c_1^n)\eta\dd x
    \quad\text{as $h\downarrow0$}.
  \end{align}
  On the other hand, thanks to the elementary inequality $|a|^2-|b|^2\le 2a\cdot(a-b)$ for vectors $a,b\in\setR^d$,
  and by the homogeneous Neumann boundary conditions satisfied by $f(c_1^n)$ and hence also by $f(\tilde c_1^h)$,
  we have that
  \begin{align*}
    \int_\Omega \frac{|\nabla f(\tilde c_1^h)|^2-|\nabla f(c_1^n)|^2}{2h}\dd x
    &\le \int_\Omega \nabla f(\tilde c_1^h)\cdot\nabla\left[\frac{f(\tilde c_1^h)-f(c_1^n)}h\right]\dd x \\
    &= -\int_\Omega \Delta f(\tilde c_1^h)\, \frac{f(\tilde c_1^h)-f(c_1^n)}h\dd x.
  \end{align*}
  On the compact support $K\subset P$ of $\eta$,
  we have $\kappa\le c_1^n\le1-\kappa$ for a suitable constant $\kappa>0$.
  We further have $\kappa/2\le\tilde c_1^h\le1-\kappa/2$ for all sufficiently small $h>0$.
  By smoothness of $f$ on $[\kappa/2,1-\kappa/2]$ thanks to Assumption \ref{asm},
  it follows that 
  \begin{align}
    \label{eq:fprimeh}
    \frac{f(\tilde c_1^h)-f(c_1^n)}h \to f'(c_1^n)\eta \quad\text{uniformly as $h\downarrow0$}.
  \end{align}
  And it further follows that also
  \begin{align}
    \label{eq:Deltafh}
    \Delta f(\tilde c_1^h) \to \Delta f(c_1^n) \quad \text{strongly in $L^2(\Omega)$ as $h\downarrow0$}
  \end{align}
  because of the following.
  We know from Lemma \ref{lem:apriori} that $f(c_1^n)$ lies in $H^2(\Omega)$,
  i.e., has square integrable first and second order derivatives.
  Again thanks to Assumption \ref{asm}, $f$ has a smooth inverse $f^{-1}$ on $[f(\kappa),f(1-\kappa)]$.
  By the chain rule for the concatenation of Sobolev functions with smooth maps,
  it follows that $c_1^n=f^{-1}(f(c_1^n))$ has square integrable first and second order weak derivatives on $K$.
  By smoothness of $\eta$,
  the first and second order derivatives of $c_1^n$ are uniformly approximated
  by the respective ones of $\tilde c_1^h$ on $K$.
  Using again the smoothness of $f$ on $[\kappa/2,1-\kappa/2]$,
  we conclude uniform approximation of $\Delta f(c_1^n)$ by $\Delta f(\tilde c_1^h)$ as $h\downarrow0$.
  Now \eqref{eq:Deltafh} follows since $\tilde c_1^h=c_1^n$ in $\Omega\setminus K$.
  Plugging \eqref{eq:chih}, \eqref{eq:fprimeh} and \eqref{eq:Deltafh} into \eqref{eq:forEL},
  we obtain in the limit $h\downarrow0$ that
  \begin{align*}
    0 \le \int_\Omega \big[\psi_1^n-\psi_2^n-f'(c_1^n)\Delta f(c_1^n)-\chi \big(c_1^n-\smallhalf\big)\big]\eta\dd x.
  \end{align*}
  The same inequality is true also for $-\eta$ in place of $\eta$, and thus is an equality.
  Since $\eta$ was an arbitrary test function with support in $P$ of zero average,
  there is a constant $A$ such that
  \begin{align*}
    \psi_1^n-\psi_2^n+A = f'(c_1^n)\Delta f(c_1^n) + \chi \big(c_1^n-\smallhalf\big)
  \end{align*}
  holds a.e.\ on $P$.
  Multiplication by $1/f'(c_1^n) = \omega(c_1^n)\omega(c_2^n)$ leads to
  \begin{align}
    \label{eq:ELintermediate2}
    \omega(c_1^n)\omega(c_2^n)[\psi_1^n-\psi_2^n+A] = \Delta f(c_1^n) + \chi \big(c_1^n-\smallhalf\big) \omega(c_1^n)\omega(c_2^n).
  \end{align}
  On the complement $\Omega\setminus P$, where either $c_1^n=0$ or $c_2^n=0$,
  the left-hand side of \eqref{eq:ELintermediate2} above vanishes a.e.\ because of $\omega(0)=0$,
  and for the same reason, the second term on the right-hand side vanishes as well.
  Also $\Delta f(c_1^n)$ vanishes a.e., because $f(c_1^n)\in H^2(\Omega)$,
  and so all of its first and second order weak partial derivatives are zero a.e.\ on the level sets \cite{Stampacchia65}.
  That is, the validity of \eqref{eq:ELintermediate2} extends from $P$ to all of $\Omega$.

  Now integrate \eqref{eq:ELintermediate2} on $\Omega$ to obtain:
  \begin{align*}
    \int_\Omega\frac{\psi_1^n-\psi_2^n+\chi \big(c_1^n-\smallhalf\big)+A}{f'(c_1^n)}\dd x
    = \int_\Omega \Delta f(c_1^n)\dd x = 0,
  \end{align*}
  where we have used that $f(c_1^n)$ satisfies homogeneous Neumann boundary conditions.
  In view of the normalization \eqref{eq:Knormal}, it follows that $A=0$.
  Finally, recalling the definition~\eqref{eq:q1} of $q_1$ and $q_2$,
  the claim of the lemma now follows from \eqref{eq:ELintermediate2}.
\end{proof}
With the consitutive relation \eqref{eq:EuLa} at hand,
we can now make the idea outlined in Section \ref{ssct:q} of the introduction rigorous
and prove $\tau$-uniform integrability of the $q_i^n$.
In the following, let
\begin{align}
  \label{eq:pd}
  p_d := \frac{d}{d-1} = 
  \begin{cases}
    2 & \text{if $d=2$}, \\
    3/2 & \text{if $d=3$}.
  \end{cases}
\end{align}
\begin{lem}\label{lem:qest}
  There is a constant $C$, only depending on the parameters of the problem,
  such that, for all $N=1,2,\ldots$,
  \begin{align}
    \label{eq:qest}
    \tau \sum_{n=1}^N\left(\|q_1^n\|_{L^{p_d}}^2+\|q_2^n\|_{L^{p_d}}^2\right)\le C(1+N\tau).
  \end{align}
\end{lem}
\begin{proof}
  We introduce the quantity
  \begin{align}
    \label{eq:LA1}
    \bar\mu^n := \frac{c_1^n\psi_1^n}{m_1\tau}+\frac{c_2^n\psi_2^n}{m_2\tau}
    = \alpha(c_1^n)q_1^n + \alpha(c_2^n)q_2^n,
  \end{align}
  where the equality follows by definition~\eqref{eq:q1} of the $q_i^n$, and since $\alpha(r)\omega(r)=r$.
  We notice further that,
  by the normalization \eqref{eq:Knormal},
  \begin{align}
    \label{eq:flat}
    \int_\Omega\bar\mu^n\dd x = 0.
  \end{align}
  Next, we recall that
  \begin{align}
    \label{eq:LA2}
    \nonl[c_1^n] = \omega(c_2^n)q_1^n-\omega(c_1^n)q_2^n
  \end{align}
  by Proposition~\ref{prp:EuLa}.
  Multiply \eqref{eq:LA1} by $\omega(c_1^n)$ and \eqref{eq:LA2} by $\alpha(c_2^n)$,
  then the sum amounts to
  \begin{align*}
    \omega(c_1^n)\bar\mu^n+\alpha(c_2^n)\nonl[c_1^n] = (c_1^n+c_2^n) q_1^n = q_1^n.
  \end{align*}
  Similarly, we obtain for $q_2^n$:
  \begin{align*}
    \omega(c_2^n)\bar\mu^n-\alpha(c_1^n)\nonl[c_1^n] = (c_1^n+c_2^n) q_2^n = q_2^n.
  \end{align*}
  Below, we show that
  \begin{align}
    \label{eq:mubarest}
    \tau\sum_{n=1}^N\big\|\bar\mu^n\big\|_{L^{p_d}}^2\le C(1+N\tau),    
  \end{align}
  which in combination with the bound \eqref{eq:F} on $\nonl[c_1^n]$, 
  and the fact that $\alpha$ and $\omega$ are bounded functions,
  yields \eqref{eq:qest}.
  
  To obtain \eqref{eq:mubarest},
  we estimate the gradient of $\bar\mu$ in $L^2(0,T;L^1(\Omega))$.
  From the definition of $\bar\mu^n$ in \eqref{eq:LA1} and the fact that $\nabla c_2^n=-\nabla c_1^n$,
  it follows that
  \begin{align}
    \label{eq:mumu}
    \nabla\bar\mu^n
    = \nabla c_1^n\left(\frac{\psi_1^n}{m_1\tau}-\frac{\psi_2^n}{m_2\tau}\right)
    + \frac{c_1^n}{m_1}\frac{\nabla\psi_1^n}\tau + \frac{c_1^n}{m_2}\frac{\nabla\psi_2^n}\tau.
  \end{align}
  We treat the two groups of terms on the right hand side separately.
  For estimation of the first term, we observe that $\omega(c_1^n)\omega(c_2^n)\nabla f(c_1^n)=\nabla c_1^n$,
  since $\omega(c_1^n)\omega(c_2^n)f'(c_1^n)=1$ on the positivity set $P:=\{0<c_1^n<1\}$,
  and both sides vanish a.e.\ on the complement $\Omega\setminus P$.
  Therefore, also recalling~\eqref{eq:LA2} again,
  \begin{align*}
    \nabla c_1^n\left(\frac{\psi_1^n}{m_1\tau}-\frac{\psi_2^n}{m_2\tau}\right)    
    = \nabla f(c_1^n)\big[\omega(c_2^n)q_1^n-\omega(c_1^n)q_2^n\big]
    = \nabla f(c_1^n)\, \nonl[c_1^n],
  \end{align*}
  hence it follows that
  \begin{align*}
    \int_\Omega\left|\nabla c_1^n\left(\frac{\psi_1^n}{m_1\tau}-\frac{\psi_2^n}{m_2\tau}\right)\right|\dd x
    \le \|\nabla f(c_1^n)\|_{L^2}\|\nonl[c_1^n]\|_{L^2}.
  \end{align*}
  The second group of terms on the right hand-side of \eqref{eq:mumu}
  is estimated by means of H\"older's inequality,
  \begin{align*}
    \int_\Omega\left| \frac{c_1^n}{m_1}\frac{\nabla\psi_1^n}\tau + \frac{c_1^n}{m_2}\frac{\nabla\psi_2^n}\tau \right|\dd x
    \le K\left[\int_\Omega\left(\frac{c_1^n}{m_1}\left|\frac{\nabla\psi_1^n}\tau\right|^2 + \frac{c_2^n}{m_2}\left|\frac{\nabla\psi_2^n}\tau\right|^2\right)\dd x\right]^{1/2},
  \end{align*}
  with a $K$ that only depends on the parameters of the problem.
  Thanks to the normalization \eqref{eq:flat},
  it follows by means of the Poincare-Wirtinger inequality that
  \begin{align*}
    \tau\sum_{n=1}^N\big\|\bar\mu^n\big\|_{L^{p_d}}^2
    \le   &\;  C\tau\sum_{n=1}^N\left[\int_\Omega|\nabla\bar\mu^n|\dd x\right]^2\dd t \\
    \le    &  \;C\sup_n\|\nabla f(c_1^n)\|_{L^2}^2 \tau\sum_{n=1}^N\|\nonl[c_1]\|_{L^2}^2 \\
          & \;+
            CK^2 \tau\sum_{n=1}^N \int_\Omega
            \left(\frac{c_1^n}{m_1}\left|\frac{\nabla\psi_1^n}\tau\right|^2 + \frac{c_2^n}{m_2}\left|\frac{\nabla\psi_2^n}\tau\right|^2\right)\dd x,
  \end{align*}
  with a constant $C$ that only depends on the geometry of $\Omega$.
  And so, recalling the estimates \eqref{eq:totalbound} on $\nabla f(c_1^n)$ in $L^2$,
  \eqref{eq:F} on $\nonl[c_1^n]$ in $L^2$, and \eqref{eq:kinetic} on the $\psi_i^n$ in a weighted $H^1$-norm,
  we arrive at~\eqref{eq:mubarest}. 
\end{proof}

\section{Convergence and conclusion of the proof of Theorem \ref{thm:main}}
\label{sct:convergence}
In this final secion, we show that the time-discrete approximations $(\bc_\tau^n)$ and $(\bq_\tau^n)$
converge to a weak solutions of the initial boundary value problem \eqref{eq:eqs}--\eqref{eq:ic}
in the sense of Theorem \ref{thm:main}.
First, introduce the usual piecewise constant interpolations in time
$\bar\bc^\tau=(\bar c_1^\tau,\bar c_2^\tau)$ and $\bar\bq^\tau=(\bar q_1^\tau,\bar q_2^\tau)$
with $\bar c_i^\tau\in L^\infty(\Omega_T)$ and $\bar q_i^\tau\in L^{p_d}(\Omega_T)$
by
\begin{align*}
  \bar c_i^\tau(t;\cdot)=c_i^n,\quad
  \bar q_i^\tau(t;\cdot)=q_i^n \quad
  \text{for all $t$ with $(n-1)\tau<t\le n\tau$}.
\end{align*}
Recall that $d=2$ or $d=3$, and the definition \eqref{eq:pd} of $p_d$.
\begin{lem}\label{lem:compact}
  There are functions
  $c_1,c_2\in L^\infty_\loc(\setR_{\ge0};H^1(\Omega))$ with $f(c_1),f(c_2)\in L^2_\loc(\setR_{\ge0};H^2(\Omega))$,
  and $q_1,q_2\in L^{p_d}(\setR_{\ge0}\times\Omega)$
  such that, for each $T>0$,
  in the limit $\tau\downarrow0$, at least along a suitable sequence,
  \begin{align}
    \label{eq:qconv}
    \bar q_i^\tau&\rightharpoonup q_i \quad \text{weakly in $L^{p_d}(\Omega_T)$}, \\
    \label{eq:cconv}
    \bar c_i^\tau&\to c_i \quad \text{strongly in $L^r(\Omega_T)$, for each $1\le r<\infty$}, \\
    \label{eq:fconv}
    \nabla f(\bar c_i^\tau)&\to \nabla f(c_i) \quad \text{strongly in $L^{24/7}(\Omega_T)$}, \\
    \label{eq:fconv2}
    f(\bar c_i^\tau) & \rightharpoonup f(c_i) \quad \text{weakly in $L^2(0,T;H^2(\Omega))$}.
  \end{align}
  Moreover, the limits $c_i$ are H\"older continuous as curves in $L^2(\Omega)$.
\end{lem}
\begin{proof}
  Ad \eqref{eq:qconv}:
  recall that \eqref{eq:qest} provides a $\tau$-uniform bound on $\bar q_i^\tau$ in $L^{p_d}(\Omega_T)$.
  Since this space is reflexive, there exist subsequences with respective weak limits.

  Ad \eqref{eq:cconv}:
  from \eqref{eq:totalbound} and the fact that $f'(r)\ge f'(1/2)>0$ thanks to Assumption \ref{asm},
  it follows for $i=1,2$, and for any $T>0$ that
  \begin{align}
    \label{eq:cboundrepeat}
    \|\bar c_i^\tau\|_{L^\infty(0,T;H^1(\Omega))}^2\dd t \le K
  \end{align}
  with a bound $K$ that might depend on $T$, but is independent of $\tau$.
  Moreover, \eqref{eq:holderL2} shows that the same sequences satisfy a uniform quasi-H\"older estimate in time,
  \begin{align}
    \label{eq:holderrepeat}
    \sup_{0<s<t<T}\big\|\bar c_i^\tau(s)-\bar c_i^\tau(t)\big\|_{L^2(\Omega)} \le C\big(\tau+|t-s|\big)^{1/4}.
  \end{align}
  We can thus invoke the generalized version of the Aubin-Lions compactness lemma from \cite[Theorem 2]{Rossi}.
  There, we choose $L^2(\Omega)$ as the base space.
  The role of the coercive integrand is played by the $H^1(\Omega)$-norm
  --- whose sublevels are clearly compact in $L^2(\Omega)$ by Rellich's theorem ---
  so that \eqref{eq:cboundrepeat} amounts to the required integral bound.
  Moreover, almost-continuity in time is guaranteed by \eqref{eq:holderrepeat}.
  We thus conclude strong convergence of the $\bar c_i^\tau$ to respective limits $c_i$ in $L^2(\Omega_T)$.
  And thanks to the uniform bound $0\le\bar c_i^\tau\le 1$,
  this implies strong convergence in \emph{any} $L^r(\Omega_T)$ with $r<\infty$.
  Moreover, these limits belong to $L^\infty(0,T;H^1(\Omega))$ by lower semi-continuity of the $H^1$-norm,
  and are H\"older continuous curves with respect to $L^2(\Omega)$,
  again thanks to \eqref{eq:cboundrepeat} and \eqref{eq:holderrepeat} above.

  Ad \eqref{eq:fconv}:
  since $f:[0,1]\to\setR$ is a continuous function, we conclude
  that also $f(\bar c_i^\tau)$ converges to the respective $f(c_i)$ in any $L^q(\Omega_T)$ with $q<\infty$.
  Further, observe that \eqref{eq:fbound} implies that
  \begin{align}
    \label{eq:fboundrepeat}
    \|f(\bar c_i^\tau)\|_{L^2(0,T;H^2(\Omega))}^2 \le K
  \end{align}
  with a constant $K$ that might depend on $T$, but not on $\tau$.
  Thanks to lower semi-continuity of the $H^2$-norm,
  it follows that $f(c_i)\in L^2(0,T;H^2(\Omega))$ satisfies the same bound \eqref{eq:fboundrepeat}. 
  We are now going to show that this implies convergence
  of $\nabla f(\bar c_i^\tau)$ to $\nabla f(c_i)$ in $L^{24/7}(\Omega_T)$.
  By the Gagliardo-Nirenberg and H\"older's inequality, we have (independently of the dimension $d$):
  \begin{align*}
    &\big\|\nabla\big[f(\bar c_i^\tau)-f(c_i)\big]\big\|_{L^{24/7}(\Omega_T)}^{24/7}
    = \int_0^T\big\|\nabla\big[f(\bar c_i^\tau)-f(c_i)\big]\big\|_{L^{24/7}(\Omega)}^{24/7}\dd t \\
    &\le C \int_0^T \big(\|f(\bar c_i^\tau)\|_{H^2(\Omega)}+\|f(c_i)\|_{H^2(\Omega)}\big)^{12/7}
      \|f(\bar c_i^\tau)-f(c_i)\|_{L^{12}(\Omega)}^{12/7}\dd t\\
    &\le C \left(\int_0^T\big[\|f(\bar c_i^\tau)\|_{H^2(\Omega)}^2+\|f(c_i)\|_{H^2(\Omega)}^2\big]\dd t\right)^{6/7}
      \left(\int_0^T\|f(\bar c_i^\tau)-f(c_i)\|_{L^{12}(\Omega)}^{12}\dd t\right)^{1/7} \\
    &\le C(2K)^{6/7} \|f(\bar c_i^\tau)-f(c_i)\|_{L^{12}(\Omega_T)}^{12/7},
  \end{align*}
  where $K$ is the bound from \eqref{eq:fboundrepeat}.
  Therefore, convergence of $f(c_i^\tau)$ carries over to convergence of $\nabla f(c_i^\tau)$.
\end{proof}
Having proven the existence of limits $\bc$ and $\bq$,
we shall now verify that these satisfy the equations \eqref{eq:wfcont} and \eqref{eq:diff_pot}.
The proof of \eqref{eq:wfcont} is divided into two steps:
in Lemma \ref{lem:wfcont1} below, we derive a discrete-in-time version of the continuity equation \eqref{eq:wfcont},
and in the subsequent Lemma \ref{lem:wfcont2}, we pass to the limit $\tau\downarrow0$.
\begin{lem}
  \label{lem:wfcont1}
  Let $\zeta\in C^\infty(\overline\Omega)$ satisfy homogeneous Neumann boundary conditions.
  Then
  \begin{align}
    \label{eq:dweak}
    \int_\Omega \zeta\frac{c_1^n-c_1^{n-1}}\tau\dd x
    = - m_1\int_\Omega q_1^n\,\big[\alpha(c_1^n)\Delta\zeta
    +\omega(1-c_1^n)\,\nabla f(c_1^n)\cdot\nabla\zeta\big]\dd x
    + \tau\err^n[\zeta],
  \end{align}
  where the error term satisfies
  \begin{align}
    \label{eq:error}
    |\err^n[\zeta]| \le \frac12\|\zeta\|_{C^2}\int_\Omega c_1^n\left|\frac{\nabla\psi_1^n}\tau\right|^2\dd x
    + \|\zeta\|_{C^0}|\Omega|.
  \end{align}
\end{lem}
\begin{proof}
  Recalling the representation \eqref{eq:push} of $c_1^{n-1}$ as push-foward of $c_1^n$,
  we obtain
  \begin{align*}
    \int_\Omega \zeta\frac{c_1^n-c_1^{n-1}}\tau\dd x
    &=\int_\Omega \zeta\frac{c_1^n-[c_1^{n-1}]_\delta}\tau\dd x
      + \int_\Omega\zeta\frac{[c_1^{n-1}]_\delta-c_1^{n-1}}\tau\dd x\\
    &=\frac1\tau\int_\Omega\big[\zeta-\zeta\circ(\id-\nabla\psi_1^n)\big]c_1^n\dd x
      + \frac\delta\tau\int_\Omega \zeta(\rho_1-c_1^{n-1})\dd x \\
    &=\int_\Omega\left[\nabla\zeta\cdot\left(\frac{\nabla\psi_1^n}\tau\right)
      +\frac{\tau}2\left(\frac{\nabla\psi_1^n}\tau\right)^T\cdot\widehat{\nabla^2\zeta}\cdot\left(\frac{\nabla\psi_1^n}\tau\right)\right]c_1^n\dd x 
      + \frac\delta\tau\int_\Omega \zeta(\rho_1-c_1^{n-1})\dd x \\
    &=\int_\Omega\nabla\zeta\cdot\left(\frac{\nabla\psi_1^n}\tau\right)c_1^n\dd x + \tau\err^n[\zeta] .
  \end{align*}
  Above, $\widehat{\nabla^2\zeta}$ is the average of the Hessian $\nabla^2\zeta$
  along the straight line segment joining $x$ to $x-\nabla\psi_1^n(x)$.
  Consequently, also using that $|c_1^{n-1}-\rho_1|\le 1$ and that $\delta\le\tau^2$ by \eqref{eq:taudelta},
  we obtain the estimate \eqref{eq:error} on $\err^n[\zeta]$.
  Now integrate by parts in the final integral above,
  \begin{align}
    \label{eq:444}
    \int_\Omega\nabla\zeta\cdot\left(\frac{\nabla\psi_1^n}\tau\right)c_1^n\dd x
    = - \int_\Omega \left[ \frac{\psi_1^n}\tau c_1^n \Delta\zeta^n+ \left(\frac{\psi_1^n}\tau\nabla c_1^n\right)\cdot\nabla\zeta\right]\dd x.
  \end{align}
  We rewrite the integral on the right-hand side.
  First, observe that
  \begin{align}
    \label{eq:lame1}
    c_1^n\psi_1^n = m_1\tau\alpha(c_1^n)q_1^n,
  \end{align}
  using on $\{c_1^n>0\}$ that $q_1^n=\omega(c_1^n)\psi_1^n/(m_1\tau)$ by definition,
  and on $\{c_1^n=0\}$ that both sides are zero, thanks to $\alpha(0)=0$.
  And second, observe that
  \begin{align}
    \label{eq:lame2}
    \nabla c_1^n\,\psi_1^n = m_1\tau\omega(c_2^n)\nabla f(c_1^n)\,q_1^n,
  \end{align}
  since $\nabla c_1^n=\omega(c_1^n)\omega(c_2^n)\nabla f(c_1^n)$
  on the positivity set $P=\{0<c_1^n<1\}$ by the fact that $f'(r)\omega(r)\omega(1-r)$ for $0<r<1$,
  and on the complement $\Omega\setminus P$ by the fact that both $\nabla c_1^n$ and $\nabla f(c_1^n)$ vanish a.e.
  Substitution of \eqref{eq:lame1}\&\eqref{eq:lame2} in \eqref{eq:444} yields \eqref{eq:dweak}.
\end{proof}
\begin{lem}
  \label{lem:wfcont2}
  For all test functions $\xi\in C^\infty_{c,n}(\setR_{>0}\times\Omega)$,
  \begin{align}
    \label{eq:13}
    \int_0^\infty\int_\Omega\left[
    -\partial_t\xi\,c_1+m_1\left(\Delta\xi\,\alpha(c_1)q_1
    + \nabla\xi\cdot\nabla f(c_1)\,\omega(1-c_1)q_1\right)
    \right]\dd x\dd t
    = 0.
  \end{align}
\end{lem}
\begin{proof}
  Introduce $\zeta^n(x):=\xi(n\tau;x)$ for $n=1,2,\ldots$,
  and the following
  piecewise constant and piecewise linear in time approximations $\xi^\tau$ and $\hat\xi^\tau$ of $\xi$,
  respectively, by:
  \begin{align*}
    \bar\xi^\tau(t;\cdot) = \zeta^n
    \quad\text{and}\quad
    \hat\xi^\tau(t;\cdot) = \frac{t-(n-1)\tau}\tau\zeta^{n+1} + \frac{n\tau-t}\tau\zeta^{n}
    \quad \text{for all $t\in((n-1)\tau,n\tau]$}.
  \end{align*}  
  Use $\zeta^n$ for $\zeta$ in \eqref{eq:dweak}, sum over $n$:
  \begin{align*}
    \tau\sum_n\err^n[\zeta_n]
    &= \tau\sum_n \int_\Omega \left[c_1^n\frac{\zeta^{n}-\zeta^{n+1}}\tau
      + \Delta\zeta^n\alpha(c_1^n)q_1^n + \nabla\zeta^n\cdot\nabla f(c_1^n)\,\omega(1-c_1^n)q_1^n \right]\dd x \\
    &= \int_0^\infty\int_\Omega \left[\bar c_1^\tau \partial_t\hat\xi^\tau
      +\Delta\bar\xi^\tau\alpha(\bar c_1^\tau)\bar q_1^\tau
      + \nabla\bar\xi^\tau\cdot\nabla f(\bar c_1^\tau)\,\omega(1-\bar c_1^\tau)\bar q_1^\tau \right]\dd x\dd t.
  \end{align*}
  We pass to the limit $\tau\downarrow0$ on both sides.
  On the left-hand side, we have thanks to \eqref{eq:error} and \eqref{eq:classical},
  \begin{align*}
    \tau\sum_{n=1}^N\big|\err^n[\zeta_n]\big|
   & \; \le \frac{m_1\|\xi\|_{C^2}}2\tau\sum_{n=1}^N\int_\Omega \left|\frac{\nabla\psi_1^n}{\tau}\right|\frac{c_1^n}{m_1}\dd x
    +\tau(N\tau)|\Omega|\|\xi\|_{C^0} \\
   &\; \le 2m_1\|\xi\|_{C^2}\nrg(\bc^0)\,\tau + \tau|\Omega_T|\big(m_1+\|\xi\|_{C^0}\big),
  \end{align*}
  which converges to zero as $\tau\to0$.
  On the right-hand side, we use that $\partial_t\hat\xi^\tau\to\partial_t\xi$ as well as
  $\nabla\bar\xi^\tau\to\nabla\xi$ and $\Delta\bar\xi^\tau\to\Delta\xi$ uniformly.
  Moreover, by \eqref{eq:cconv}, and since $\alpha,\omega:[0,1]\to\setR$ are continuous,
  we have in particular that
  \begin{align*}
    \alpha(\bar c_1^\tau)\to\alpha(c_1)
    \quad\text{and}\quad
    \omega(1-\bar c_1^\tau) \to \omega(1-c_1)
  \end{align*}
  in $L^{24}(\Omega_T)$.
  In view of \eqref{eq:qconv} and \eqref{eq:fconv}, 
  \begin{align*}
    \bar q_1^\tau\nabla f(\bar c_1^\tau) \rightharpoonup q_1\nabla f(c_1)
    \quad \text{in $L^{24/23}(\Omega_T)$}.
  \end{align*}
  Therefore, the integral converges.
\end{proof}
The purpose of the next and final lemma is to derive the constitutive equations~\eqref{eq:c1+c2=1_cont} and \eqref{eq:diff_pot}.
\begin{lem}
  \label{lem:easyconv}
  Let $c_i$ and $q_i$ be as in Lemma~\ref{lem:compact},
  then $c_1+c_2 = 1$ and
  \begin{align}
    \label{eq:diffpotrepeat}
    \omega(c_1)\,q_2 - \omega(c_2)\,q_1 = \Delta f(c_1) + \chi \textstyle{\left(c_1-\smallhalf \right)}\omega(c_1)\omega(c_2).    
  \end{align}
\end{lem}
\begin{proof}
  Because of~\eqref{eq:c1+c2=1}, we have $\bar c_1^\tau +\bar c_2^\tau =1$,
  which clearly yields $c_1+c_2 = 1$ in the limit, using the strong convergence from~\eqref{eq:cconv}.

  Next, recall that \eqref{eq:EuLa} is precisely \eqref{eq:diffpotrepeat},
  with $\bar\bc^\tau$ in place of $\bc$, and with $\bar\bq^\tau$ in place of $\bq$,
  i.e.,
  \begin{align}
    \label{eq:diffpotrepeat2}
    \omega(\bar c_1^\tau)\,\bar q_2^\tau - \omega(\bar c_2^\tau)\,\bar q_1^\tau
    = \Delta f(\bar c_1^\tau) + \chi \textstyle{\left(\bar c_1^\tau-\smallhalf\right)}\omega(\bar c_1^\tau)\omega(\bar c_2^\tau).    
  \end{align}
  By the strong convergence~\eqref{eq:cconv} of $\bar c_i^\tau$
  and thanks to the continuity of $\omega$,
  it follows that $\omega(\bar c_i^\tau)$ converges to $\omega(c_i)$ strongly in, say, $L^3(\Omega_T)$.
  In combination with the weak convergence~\eqref{eq:qconv} of the $\bar q_i^\tau$,
  we obtain weak convergence of the products,
  \begin{align*}
    \omega(\bar c_1^\tau)\,\bar q_2^\tau  \rightharpoonup \omega(c_1) q_2 
    \quad \text{and}\quad
    \omega(\bar c_2^\tau)\,\bar q_1^\tau  \rightharpoonup \omega(c_2) q_1
  \end{align*}
  in $L^1(\Omega)$.
  Trivially, also
  \[
    \textstyle{\left(\bar c_1^\tau-\smallhalf\right)}\omega(\bar c_1^\tau)\omega(\bar c_2^\tau)
    \to 
    \textstyle{\left(c_1-\smallhalf\right)}\omega(c_1)\omega(c_2) 
  \]
  strongly in $L^1(\Omega_T)$.
  Finally, weak convergence $\Delta f(\bar c_1^\tau)\rightharpoonup \Delta f(\bar c_1)$ in $L^2(\Omega_T)$
  is implied by \eqref{eq:fconv2}.
  We thus obtain \eqref{eq:diffpotrepeat} as limit of \eqref{eq:diffpotrepeat2}.
\end{proof}
The proof of Theorem \ref{thm:main} is a conclusion of Lemma \ref{lem:wfcont2} and Lemma \ref{lem:easyconv}.


\appendix

\section{}
%
%
\begin{lem}
  \label{lem:mollify}
  There is a constant $K$,
  expressible in terms of $\rho_1/m_1$, $\rho_2/m_2$, and geometric properties of $\Omega$,
  such that for all $\bc\in\spc$:
  \begin{align}
    \label{eq:mollify}
    \dst(\bc,[\bc]_\delta)^2\le K\delta.
  \end{align}
  Consequently, for any $\bc,\bar\bc\in\spc$:
  \begin{align}
    \label{eq:triangle}
    \dst(\bc,\bar\bc)^2
    \le 2\dst\big(\bc,[\bar\bc]_\delta\big)^2+2K\delta.    
  \end{align}
\end{lem}
\begin{proof}
  Define a (sub-optimal) transport plan $\gamma$ from $c_i$ to $[c_i]_\delta=(1-\delta)c_i+\delta\rho_i$
  as follows:
  \begin{align*}
    \gamma = (1-\delta)(\id,\id)\#(c_i\lbg_\Omega) + \frac{\delta}{|\Omega|} (c_i\lbg_\Omega)\otimes\lbg_\Omega.
  \end{align*}
  The marginals are as desired, i.e., for any $\xi,\eta\in C(\Omega)$,
  we have that
  \begin{align*}
    \iint_{\Omega\times\Omega} \xi(x)\dd \gamma(x,y)
    &= (1-\delta)\int_\Omega\xi(x)c_i(x)\dd x + \delta\int_\Omega\xi(x)c_i(x)\dd x \fint_\Omega\dd y
      = \int_\Omega\xi(x)\eta(x)\dd x, \\
    \iint_{\Omega\times\Omega} \eta(y)\dd \gamma(x,y)
    &= (1-\delta)\int_\Omega\eta(y)c_i(y)\dd x + \delta \fint_\Omega c_i(x)\dd x \int_\Omega\eta(y)\dd y
      = \int_\Omega \eta(y)\big[(1-\delta)c_i(y)+\delta\rho_i\big]\dd y.
  \end{align*}
  And the corresponding costs amount to
  \begin{align*}
    \iint_{\Omega\times\Omega}|x-y|^2\dd\gamma(x,y)
    &= \frac\delta{|\Omega|}\iint_{\Omega\times\Omega} |x-y|^2c_i(x)\dd x\dd y \\
    & \le \frac{\delta\,\operatorname{diam}(\Omega)^2}{|\Omega|}\int_\Omega c_i(x)\dd x\int_\Omega\dd y
    =\delta\,\operatorname{diam}(\Omega)^2|\Omega|\rho_i.
  \end{align*}
  In summary,
  \begin{align*}
    \dst(\bc,[\bc]_\delta)^2 = \frac{\wass(c_1,[c_1]_\delta)^2}{m_1} + \frac{\wass(c_2,[c_2]_\delta)^2}{m_2}
    \le \delta\,\operatorname{diam}(\Omega)^2\left(\frac{|\Omega|\rho_1}{m_1}+\frac{|\Omega|\rho_2}{m_2}\right).
  \end{align*}
  The inequality \eqref{eq:triangle} now follows from the triangle inequality,
  that is inherited from $\wass$ to $\dst$,
  \begin{align*}
    \dst(\bc,\bar\bc)^2\le2\dst\big(\bc,[\bar\bc]_\delta\big)^2 + 2\dst\big(\bar\bc,[\bar\bc]_\delta\big)^2,
  \end{align*}
  in combination with \eqref{eq:mollify}.
\end{proof}
\begin{lem}
  \label{lem:fromW2toL2}
  For all $\bc,\bc'\in\spc$ with $c_i,c_i'\in H^1(\Omega)$ and $c_1+c_2\equiv1\equiv c_1'+c_2'$,
  \begin{align}
    \|\bc'-\bc\|_{L^2}^2\le 2\sqrt{m_1}\big(\|\nabla c_1\|_{L^2}+\|\nabla c_1'\|_{L^2}\big)\dst(\bc',\bc).
  \end{align}
\end{lem}
\begin{proof}
  Let $(\varphi_1,\psi_1)$ be a pair of Kantorovich potentials for the optimal transport from $c_1$ to $c_1'$.
  For each $s\in[0,1]$, define $T_s:\Omega\to\Omega$ by $T_s(x) = x-s\nabla\varphi_1(x)$.
  For any test function $\zeta\in C^1(\overline\Omega)$,
  \begin{align*}
    \int_\Omega [c_1'-c_1]\zeta\dd x
    &= \int_\Omega \big[\zeta\circ T_1-\zeta\big]c_1\dd x \\
    &= \int_\Omega \left[\int_0^1 \nabla\zeta\circ T_s\cdot\nabla\varphi\dd s\right]\,c_1\dd x \\
    &\le\int_0^1\left(\int_\Omega|\nabla\zeta|^2\circ T_s\,c_1\dd x\right)^{1/2}
      \left(\int_\Omega|\nabla\varphi_1|^2c_1\dd x\right)^{1/2}\dd s \\
    &= \int_0^1\left(\int_\Omega|\nabla\zeta|^2T_s\#c_1\dd x\right)^{1/2}\dd s\, \wass(c_1,c_1').
  \end{align*}
  Using the fact that
  \begin{align*}
    \sup_\Omega T_s\#c_1 \le \max\left(\sup_\Omega c_1,\sup_\Omega c_1'\right) = 1,
  \end{align*}
  it follows that
  \begin{align*}
    \int_\Omega \big|[c_1'-c_1]\zeta\big|\dd x\le \|\nabla\zeta\|_{L^2}\wass(c_1,c_1'),
  \end{align*}
  and consequently --- using for $\zeta$ approximations of $c_1'-c_1$ in $C^1$ ---
  \begin{align*}
    \|c_1'-c_1\|_{L^2}^2\le \|\nabla(c_1'-c_1)\|_{L^2}\wass(c_1,c_1')
    \le \big(\|\nabla c_1'\|_{L^2}+\|\nabla c_1\|_{L^2}\big)\wass(c_1,c_1').
  \end{align*}
  By hypothesis, $c_1-c_1'=c_2'-c_2$.
  Thus, recalling the definition of $\dst$,
  we obtain
  \begin{align*}
    \|\bc'-\bc\|_{L^2}^2
    &=  2\|c_1'-c_1\|_{L^2}^2 \\
    &\le 2\sqrt{m_1}\big(\|\nabla c_1\|_{L^2}+\|\nabla c_1'\|_{L^2}\big)^2
      \left(\frac{\wass(c_1',c_1)^2}{m_1}\right)^{1/2} \\
    &\le 2\sqrt{m_1}\big(\|\nabla c_1\|_{L^2}+\|\nabla c_1'\|_{L^2}\big)\dst(\bc',\bc).
  \end{align*}
\end{proof}
%

\bibliographystyle{plain}
\bibliography{biblio}

\end{document}